\documentclass[leqno, 12pt]{article}
\usepackage{amsmath,amsfonts,amsthm,amssymb,indentfirst}
\usepackage{xy} \xyoption{all}
\usepackage[colorlinks]{hyperref}

\hypersetup{linkcolor=blue, urlcolor=blue, citecolor=red}

\newtheorem{theorem}{Theorem}
\newtheorem{lemma}[theorem]{Lemma}
\newtheorem{corollary}[theorem]{Corollary}
\newtheorem{proposition}[theorem]{Proposition}
\newtheorem{question}[theorem]{Question}

\theoremstyle{definition}
\newtheorem{definition}[theorem]{Definition}
\newtheorem{notation}[theorem]{Notation}
\newtheorem{example}[theorem]{Example}
\newtheorem{remark}[theorem]{Remark}

\newcommand{\Spec}{\mathrm{Spec}}
\newcommand{\pth}{\mathrm{Path}}
\newcommand{\so}{\mathbf{s}}
\newcommand{\ra}{\mathbf{r}}
\newcommand{\MT}{\mathcal{M}}
\newcommand{\RT}{\mathcal{R}}
\newcommand{\AT}{\mathcal{A}}
\newcommand{\ATC}{\mathcal{AC}}
\newcommand{\R}{\mathbb{R}}

\newcommand{\Z}{\mathbb{Z}}

\newcommand{\N}{\mathbb{N}}

\begin{document}

\title{Realizing posets as prime spectra\\ of Leavitt path algebras}
\author{G.\ Abrams, G.\ Aranda Pino, Z.\ Mesyan, and C.\ Smith}

\maketitle

\begin{abstract}
We associate in a natural way to any partially ordered set $(P,\leq)$ a directed graph $E_P$ (where the vertices of $E_P$ correspond to the elements of $P$, and the edges of $E_P$ correspond to related pairs of elements of $P$), and then describe the prime spectrum $\Spec(L_K(E_P))$ of the resulting Leavitt path algebra $L_K(E_P)$. This construction allows us to realize a wide class of partially ordered sets as the prime spectra of rings. More specifically, any partially ordered set in which every downward directed subset has a greatest lower bound, and where these greatest lower bounds satisfy certain compatibility conditions, can be so realized. In particular, any partially ordered set satisfying the descending chain condition is in this class.   
\medskip




\end{abstract}

For a ring $R$, the partially ordered set $(\Spec(R), \subseteq)$, consisting of the prime ideals of $R$ under set inclusion, has been the focus of significant research attention for many decades. While classically such investigations dealt only with commutative unital rings (e.g.,~\cite{Facchini, HW, Hochster, Lewis, LO, Speed, Wiegand}), recent energy has been spent studying more general classes (e.g.,~\cite{GRV, Sarussi}). Specifically, one may ask whether there are any necessary conditions on the partially ordered set $(\Spec(R), \subseteq )$ for general $R$.  More precisely, one may pose a ``Realization Question":  given a poset $(P, \leq)$, does there exist a ring $R$ for which $(\Spec(R), \subseteq ) \cong (P,\leq)$? In this article we answer such a question in the affirmative for a large class of partially ordered sets, thereby extending the class for which  an affirmative answer was heretofore known.  In the process, we are led to a naturally occurring collection of partially ordered sets, an until-now-unidentified collection which seems to be of interest in its own right. The class of algebras by which we realize these posets as prime spectra are built as Leavitt path algebras. Such algebras also arise in this Realization Question context, as they have at their heart a directed graph, and there is a natural way (germane to the Realization Question, and described herein) to associate a directed graph $E_P$ to any given poset $P$.  

The article is organized as follows. We start by reminding the reader of various  definitions and properties of partially ordered sets, as well as of the definition of Leavitt path algebras. After providing a necessary general condition on $\mathrm{Spec}(R)$ for arbitrary $R$, which will give some context to our results (Property GLB, Proposition~\ref{primeintersect}), we present the definition of the directed graph $E_P$ arising from a partially ordered set $P$. Subsequently, we analyze the relationship between $(P, \leq)$ and $(\Spec(L_K(E_P)), \subseteq)$, where $L_K(E_P)$ denotes the Leavitt path algebra of $E_P$ with coefficients in a field $K$. We show (Proposition \ref{dccthrm}) that the descending chain condition on $P$ is necessary and sufficient in order that the easily-anticipated map from $P$ to $\Spec(L_K(E_P))$ is an order-isomorphism. In our main Realization Question result (Theorem \ref{EPspec}), we show that $\Spec(L_K(E_P))$ is isomorphic to a partially ordered set $\AT(P)$ built from $P$ by appropriately adding greatest lower bounds to certain subsets of $P$. We then proceed to show (Theorem~\ref{APprop}) that  the construction of the poset $\AT(P)$ from $P$ may be interpreted quite naturally in purely poset-theoretic terms.  Specifically, for any poset $P$ we define the poset $\RT(P)$, by removing greatest lower bounds from certain subsets of $P$, and then show that $P \cong \AT(Q)$ for some poset $Q$ precisely when $\RT(\AT(P)) \cong \AT(\RT(P))$. We also provide in Theorem \ref{APprop} three conditions on a poset $P$ which, taken together, are necessary  and sufficient to ensure that $P \cong \AT(Q)$. The remainder of the article is spent discussing how various  of these three conditions play out in other contexts. In particular, we show how these conditions relate to a currently open question (motivated by a classical result of Kaplansky) regarding  the structure of  the prime spectrum of an arbitrary ring.  

\bigskip

\noindent
\textbf{\large{Acknowledgments}}  
   
\smallskip

\normalsize

We are grateful to George Bergman for showing us Theorem~\ref{bergman} and allowing us to include it here. We are also grateful to Be'eri Greenfeld for extremely helpful correspondence regarding the Realization Question and the structure of $\Spec(R)$, for arbitrary rings $R$. In addition, we thank the referee for a close reading of the manuscript.

The first  author is partially supported by a Simons Foundation Collaboration Grants for Mathematicians Award \#208941.  

The second author was partially supported by the Spanish MEC through project  MTM\break2013-41208-P and by the Junta de Andaluc\'{\i}a and Fondos FEDER, jointly, through projects FQM-336 and FQM-7156.

\section{Partially ordered sets, directed graphs, and Leavitt path algebras}\label{IntroSection}

In this initial section we give an overview of the key foundational  ideas used in this article.   

\bigskip

\noindent
\textbf{\large{Partially ordered sets}}  
\normalsize
\smallskip

A  \emph{preordered set} $(P, \leq)$ is a set $P$ together with a binary relation $\leq$ which is reflexive  and transitive.   If in addition $\leq$ is antisymmetric, then $(P,\leq)$ is called a \emph{partially ordered set} (often shortened simply to \emph{poset}). If a poset $(P, \leq)$ has the property that $p\leq q$ or $q \leq p$ for all $p,q \in P$, then $(P, \leq)$ is called a  \emph{totally ordered set}, or alternatively, a \emph{chain}.  

Let $(P, \leq)$ be a partially ordered set. An element $x$ of  $P$  is called 
 
\qquad -  a \emph{minimal} element of $P$ if there is no  $y \in P$ such that $y < x$, and 
  
\qquad - a \emph{least} element of $P$ if $x \leq y$ for all $y \in P$.   
    
\noindent
Clearly if $P$ contains a least element then it is necessarily unique. $P$ is called an \emph{infinite descending chain} if $P$ is totally ordered, and $P$ has no minimal elements (if and only if $P$ has no least element). $P$ is \emph{downward directed} if it is nonempty, and for all $p,q \in P$ there exists $r \in P$ such that $p \geq r$ and $q\geq r$. If $P$ is downward directed and contains a minimal element, then it is necessarily the least element of $P$.   

 A \emph{Hasse diagram} of $(P, \leq)$ is a diagram (if one exists) representing the elements of the set $P$ as points, with a line drawn upwards from a point $x$ to a point $y$ whenever $x < y$ and there is no element $z\in P$ satisfying $x < z < y$.

Let $(P_1, \leq_1)$ and $(P_2, \leq_2)$ be two partially ordered sets, and let $f:P_1 \to P_2$ be a function. Then $f$ is \emph{order-preserving} if $x\leq_1 y$ implies that $f(x) \leq_2 f(y)$ for all $x,y \in P_1$. Also $f$ is \emph{order-reflecting} if $f(x) \leq_2 f(y)$ implies that $x \leq_1 y$ for all $x, y \in P_1$. In this case, $f$ is necessarily injective, since $f(x)=f(y)$ implies that $x \leq_1 y$ and $y \leq_1 x$. If $f$ is both order-preserving and order-reflecting, then it is an \emph{order-embedding}. If $f$ is an order-embedding and bijective, then it is an \emph{order-isomorphism}. We write $(P_1, \leq_1) \cong (P_2, \leq_2)$ if there is an order-isomorphism between the two partially ordered sets.

For any poset $(P, \leq)$ and subset $S$ of $P$, by restriction $(S, \leq)$ is a poset. We say that $(P, \leq)$ satisfies the \emph{descending chain condition} (shortened as  \emph{DCC}) if there is no subset $S$ of $P$ for which $(S, \leq)$ is  an infinite descending chain. A \emph{lower bound} for $S \subseteq P$ is an element $x$ of $P$ such that $x\leq s$ for all $s\in S$. A \emph{greatest lower bound} of $S$ is a lower bound $x$ of $S$ with the additional property that $y\leq x$ for every $y\in P$ having $y\leq s$ for all $s\in S$. Easily if such a greatest lower bound $x$ of $S$ is an element of $S$ itself, then $x$ is the least element of $S$. Just as easily, if a greatest lower bound for $S$ exists, it is necessarily unique.   If $(P, \subseteq)$ is a set of sets ordered by inclusion, and $S \subseteq P$ is a collection for which $x = \bigcap_{s\in S} \hspace{.02in}  s \in P$, then necessarily $x $  is the greatest lower bound of $S$.  
  
The following three subsets of $\Z$ (the partially ordered set of integers) will be of importance to us: $\N = \{0,1,2, ...\}$, $\Z^+ = \{1,2,3, ...\}$, and $\Z^- = \{-1,-2,-3,...\}$. 

The cardinality of a set $X$ is denoted by $|X|$.  
 
\begin{remark}\label{dccvsglb}
If a poset $P$ has DCC,  then easily every downward directed subset $S$ of $P$ (in particular, every chain in $P$) necessarily has a greatest lower bound in $P$ (indeed, has a least element in $S$). However, there are posets $P$ where every downward directed subset has a greatest lower bound in $P$, but which do not have DCC. The poset $P = \{0\} \cup \{ \frac{1}{n} \ | \ n \in \Z^+ \}$ with the usual order is such. 
\hfill $\Box$
\end{remark}

\bigskip

\noindent
\textbf{\large{Directed graphs}}  
\normalsize
\smallskip

A \emph{directed graph} $E=(E^0,E^1,\so, \ra)$ consists of two  sets $E^0,E^1$ (the elements of which are called \emph{vertices} and \emph{edges}, respectively), together with functions $\so,\ra:E^1 \to E^0$, called \emph{source} and \emph{range}, respectively. We shall refer to directed graphs as simply ``graphs" from now on. A \emph{path} $p$ in $E$ is a finite sequence $e_1\cdots e_n$ of edges $e_1,\dots, e_n \in E^1$ such that $\ra(e_i)=\so(e_{i+1})$ for $i\in \{1,\dots,n-1\}$. Here we define $\so(p):=\so(e_1)$ to be the \emph{source} of $p$ and $\ra(p):=\ra(e_n)$ to be the \emph{range} of $p$. We view the elements of $E^0$ as paths of length $0$ (extending $\so$ and $\ra$ to $E^0$ via $\so(v)=v = \ra(v)$ for all $v\in E^0$), and denote by $\pth(E)$ the set of all paths in $E$. A path $p = e_1\cdots e_n$ is said to be \emph{closed} if $\so(p)=\ra(p)$. Such a path is said to be a \emph{cycle} if the path has nonzero length, and in addition $\so(e_i)\neq \so(e_j)$ for every $i\neq j$. A cycle consisting of just one edge is called a \emph{loop}. A graph which contains no cycles is called \emph{acyclic}.

A vertex $v \in E^0$ for which the set $\so^{-1}(v) = \{e\in E^1 \mid \so(e)=v\}$ is finite is said to have \emph{finite out-degree}. A graph $E$ is said to have \emph{finite out-degree}, or to be \emph{row-finite}, if every vertex of $E$ has finite out-degree. A vertex $v \in E^0$ such that $\so^{-1}(v) = \emptyset$ is called a \emph{sink}, while a vertex having finite out-degree which is not a sink is called \emph{regular}. If $u,v \in E^0$ are distinct vertices such there is a path $p \in \pth(E)$ satisfying $\so(p) = u$ and $\ra(p)=v$, then we write $u > v$. It is easy to see that $(E^0, \geq)$ is a preordered set. Given a vertex $v \in E^0$, set $\MT(v) = \{w \in E^0 \mid w \geq v\}$. A subset $H$ of $E^0$ is \emph{hereditary} if whenever $u \in H$ and $u \geq v$ for some $v \in E^0$, then $v \in H$. Also $H \subseteq E^0$ is \emph{saturated} if $\ra(\so^{-1}(v)) \subseteq H$ implies that $v \in H$ for any regular $v \in E^0$. A nonempty subset $M$ of $E^0$ is a \emph{maximal tail} if it satisfies the following conditions.

(MT1) If $v\in M$ and $u\in E^0$ are such that $u \geq v$, then $u\in M$.

(MT2) For every regular $v \in M$ there exists $e \in E^1$ such that $\so(e)=v$ and $\ra(e) \in M$. 

(MT3) For all $u,v \in M$ there exists $w\in M$ such that $u \geq w$ and $v \geq w$.\\
For any subset $H \subseteq E^0$ it is easy to see that $H$ is hereditary if and only if $M=E^0\setminus H$ satisfies MT1, and $H$ is saturated if and only if $M=E^0\setminus H$ satisfies MT2.

\bigskip

\noindent
\textbf{\large{Leavitt path algebras}}  
\normalsize
\smallskip

For any field $K$ and (nonempty) graph $E$ there are a number of well-studied ways in which one may produce an associative $K$-algebra which reflects the structure of $E$. In the current context, the Leavitt path algebra provides us with an effective tool. Precisely, the \emph{Leavitt path $K$-algebra} $L_K(E)$ \emph{of $E$} is the $K$-algebra generated by the set $\, \{v\mid v\in E^0\} \cup \{e,e^*\mid e\in E^1\}$, subject to the following relations:

\smallskip

{(V)} \ \ \ \  $vw = \delta_{v,w}v$ for all $v,w\in E^0$,

{(E1)} \ \ \ $\so(e)e=e\ra(e)=e$ for all $e\in E^1$,

{(E2)} \ \ \ $\ra(e)e^*=e^*\so(e)=e^*$ for all $e\in E^1$,

{(CK1)}  \hspace{.01in}  $e^*f=\delta _{e,f}\ra(e)$ for all $e,f\in E^1$, and

{(CK2)} \  $v=\sum_{e\in \so^{-1}(v)} ee^*$ for every regular vertex $v\in E^0$.   

\smallskip

The Leavitt path algebra $L_K(E)$ may also be viewed as follows.  Given a (directed) graph $E$, let $\widehat{E}$ denote the directed graph for which $\widehat{E}^0 = E^0$ and $\widehat{E}^1 = E^1 \sqcup (E^1)^*$, where $(E^1)^* = \{e^* \mid e\in E^1\}$ and for each $e^* \in (E^1)^*$, $\so(e^*) = \ra(e)$ and $\ra(e^*) = \so(e)$.   (Throughout, $\sqcup$ denotes ``disjoint union.")   That is, $\widehat{E}$ is constructed from $E$ by adding to $E^1$, for each edge $e$, a new edge $e^*$ having  orientation opposite to that of $e$.   Then  $L_K(E)$ is the quotient of the standard path algebra $K\widehat{E}$, modulo the relations CK1 and CK2.   

For all $v \in E^0$ we define $v^*:=v$, and for all paths $\alpha  = e_1 \cdots e_n$ ($e_1, \dots, e_n \in E^1$) we set $\alpha^*:=e_n^* \cdots e_1^*$, $\ra(\alpha^*):=\so(\alpha)$, and $\so(\alpha^*):=\ra(\alpha)$. With this notation, every element of $L_K(E)$ can be expressed (though not necessarily uniquely) in the form $\sum_{i=1}^n k_i\alpha_i\beta_i^*$ for some $k_i \in K$ and $\alpha_i,\beta_i \in \pth (E)$.

It is easily established that $L_K(E)$ is unital (with multiplicative identity $\sum_{v\in E^0}v$) if and only if $E^0$ is a  finite set.    

For additional information and background on Leavitt path algebras see e.g., \cite{ASurvey}.

\section{Prime ideals}\label{primessection}

Recall that given a (not necessarily unital) ring $R$, a proper ideal $I$ of $R$ is \emph{prime} if for all $x,y \in R$, $xRy \subseteq I$ implies that either $x \in I$ or $y \in I$ (equivalently, if for all ideals $A,B$ of $R$, $AB\subseteq I$ implies that either $A \subseteq I$ or $B \subseteq I$).  The set of prime ideals of $R$, denoted $\Spec(R)$, is clearly a poset under set inclusion.   In~\cite[Theorems 9 and 11]{Kap} Kaplansky considers the question of which partially ordered sets can arise as $\Spec(R)$, where $R$ is a  commutative unital ring. We summarize those results here.    

\begin{theorem}[Kaplansky]  \label{Kapthrm}
Let $R$ be a commutative unital ring.   
\begin{enumerate}
\item[$(1)$] If $\, \{I_s \mid s \in S\}$ is a totally ordered set of prime ideals of $R$, then  $\, \bigcap_{s \in S} I_s$ is a prime ideal of $R$. $($Rephrased:  every totally ordered subset of $\, \Spec(R)$ has a greatest lower bound in $\, \Spec(R)$.$)$ 
\item[$(2)$] If $\, \{I_s \mid s \in S\}$ is a totally ordered set of prime ideals of $R$, then $\, \bigcup_{s \in S} I_s$ is a prime ideal of $R$.
\item[$(3)$] If $I \subset J$ are distinct prime ideals of $R$, then there exist distinct prime ideals $I'$ and $J'$ of $R$ such that $I \subseteq I' \subset J' \subseteq J$, and there are no prime ideals lying properly between $I'$ and $J'$.
\end{enumerate}
\end{theorem}

Theorem~\ref{Kapthrm} can be established nearly verbatim, with essentially the same proofs as given in~\cite{Kap},  for not-necessarily-unital commutative rings.  (Statement (2) must be modified to include the possibility in this generality that the union of a totally ordered set of prime ideals might yield the non-prime ideal $R$ of $R$.) The three obvious questions regarding the generalizations of the three assertions of Theorem~\ref{Kapthrm} to noncommutative rings yield three different  answers.  

\medskip

Statement  (1) extends verbatim (and more generally): see Proposition~\ref{primeintersect}.

\smallskip

Statement (2) does not extend: for instance,  Bergman has constructed a noncommutative (unital) ring having a chain of prime ideals whose union is not prime (see~\cite[Example 4.2]{Passman} or~\cite[Example 2.1]{GRV}). A Leavitt path algebra with this property is given in Example~\ref{union-eg} below.

\smallskip

Statement (3): whether or not Statement (3) extends to all rings is currently not known. (We discuss this question in Section~\ref{closingremarks}; a nice account of various aspects of this question is given in~\cite{Sarussi}.)  

\medskip

Here is the extension of Statement (1) of Theorem~\ref{Kapthrm} to noncommutative rings and more general collections of prime ideals.

\begin{proposition}\label{primeintersect}
Let $R$ be any ring, and let $\, \{I_s \mid s \in S\}$ be a downward directed collection of prime ideals of $R$; that is,  for all $s,t \in S$ there exists $r \in S$ such that $I_r \subseteq I_s \cap I_t$. Then $I= \bigcap_{s \in S} I_s$ is a prime ideal. Consequently, if $(P,\leq)$ is a partially ordered set that can be represented as the prime spectrum of a ring, then every downward directed subset of $P$ has a greatest lower bound.     
\end{proposition}

\begin{proof}
Since all the $I_s$ are proper ideals of $R$, so is $I$. Now, let $x,y \in R$ be any elements, and suppose that $xRy \subseteq I$ but $x \notin I$. Then $x \notin I_s$ for some $s \in S$, and hence $y \in I_s$, as $I_s$ is prime. By hypothesis, for any $t \in S$ there exists $r \in S$ such that $I_r \subseteq I_s \cap I_t$. Since $I_r \subseteq I_s$, we have $x \notin I_r$ but $xRy \subseteq I_r$. Thus $y \in I_r$, and hence also $y \in I_t$. Since $t \in S$ was arbitrary, we conclude that $y \in I$, and therefore $I$ is prime.

The final claim follows from the fact that $\bigcap_{s \in S} I_s$ is necessarily the greatest lower bound in $\Spec(R)$ of the downward directed set $\{I_s \mid s \in S\}$.
\end{proof}

We will denote  by GLB  the property of $\Spec(R)$ described in Proposition~\ref{primeintersect}, specifically that every downward directed subset has a greatest lower bound. In Theorem~\ref{bergman} we will show that for any poset $P$, having the property GLB is equivalent to having the property that every totally ordered subset of $P$ has a greatest lower bound. Thus the second claim in Proposition~\ref{primeintersect} can also be proved by first noting that Statement (1) of Theorem~\ref{Kapthrm} holds for noncommutative rings $R$, and then applying the purely poset-theoretic Theorem~\ref{bergman}.

\medskip

The structure of $\Spec(L_K(E))$ for a finite graph $E$ was described in \cite{APPSM}. Subsequently, an explicit description of the prime ideals of $L_K(E)$ for an arbitrary graph $E$ was given by Rangaswamy in~\cite[Theorem 3.12]{Ranga}. We present this result now, once a modicum of additional notation has been established. In a graph $E$, a cycle $c \in \pth(E)$ is said to be \emph{WK} (for ``without Condition (K)") if no vertex along $c$ is the source of another distinct cycle in $E$ (i.e., one possessing a different set of edges). A vertex $v \in E^0$ is called a \emph{breaking vertex} of a hereditary saturated subset $H$ of $E^0$ if $v \in E^0\setminus H$, $|\so^{-1}(v)| \geq \aleph_0$, and $1 \leq |\so^{-1}(v) \cap \ra^{-1}(E^0\setminus H)| <\aleph_0$. The set of all breaking vertices of $H$ is denoted by $B_H$. Given $v \in B_H$, we define  $v^H \in L_K(E)$ by setting $v^H  = v - \sum_{\so(e)=v, \ra(e) \notin H} ee^*$.

\begin{theorem} [Rangaswamy] \label{primeclass}
Let $E$ be a graph, $K$ a field, $I$ a proper ideal of $L_K(E)$, and $H = I \cap E^0$. Then $I$ is a prime ideal if and only if $I$ satisfies one of the following conditions.
\begin{enumerate}
\item[$(1)$] $I = \langle H \cup \{v^H \mid v \in B_H\} \rangle$ $($where $B_H$ may be empty$)$ and $E^0\setminus H$ satisfies \emph{MT3}. 
\item[$(2)$] $I = \langle H \cup \{v^H \mid v \in B_H\setminus \{u\}\}\rangle$ for some $u \in B_H$ and $E^0\setminus H = \MT(u)$.
\item[$(3)$] $I = \langle H \cup \{v^H \mid v \in B_H\} \cup \{f(c)\}\rangle$ where $c \in \pth(E) \setminus E^0$ is a WK cycle having source $u \in E^0$, $E^0\setminus H = \MT(u)$, and $f(x)$ is an irreducible polynomial in $K[x, x^{-1}]$.
\end{enumerate}
\end{theorem}

\begin{remark}\label{HcupstuffcapE0equalsH}
It is easy to see that $H = I \cap E^0$ is hereditary and saturated for any ideal $I \subseteq L_K(E)$. Also, it is noted in \cite[Section 2]{Ranga} that 
$$\langle H \cup \{v^H \mid v \in B_H\} \rangle \cap E^0 = H$$
for any hereditary and saturated subset $H$ of $E^0$.  \hfill $\Box$
\end{remark}

Theorem~\ref{primeclass}  provides us with  all the information we need to  construct a Leavitt path algebra containing a chain of prime ideals whose union is not prime.

\begin{example} \label{union-eg}
Let $K$ be any field, and let $E$ be the following graph.
$$
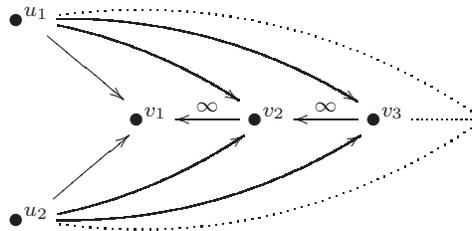
\xymatrix{ 
{\bullet}^{u_1} \ar[rd] \ar@/^.5pc/[rrd] \ar@/^1pc/[rrrd] \ar@/^1.5pc/@{.}[rrrrd]
& & & & \\
 & {\bullet}^{v_1} &  {\bullet}^{v_2} \ar[l]_\infty &  {\bullet}^{v_3} \ar[l]_{\infty} & \ar@{.}[l] \\
{\bullet}^{u_2} \ar[ru] \ar@/^-.5pc/[rru] \ar@/^-1pc/[rrru] \ar@/^-1.5pc/@{.}[rrrru]
& & & &
}$$
(The symbol $\infty$ appearing adjacent to an edge indicates that there are countably infinitely many edges from one vertex to the other.)  Then the two sets $\{u_1\}$ and $\{u_2\}$ are maximal tails of $E$, and all the other maximal tails of $E$ are of the form $M_n = \{u_1, u_2\} \cup \{v_i \mid i \geq n\}$ ($n \in \Z^+$). 

In particular, it follows from Theorem~\ref{primeclass}(1) that $\langle E^0 \setminus M_n \rangle = \langle \{v_1, \dots, v_{n-1}\} \rangle$ is a prime ideal in $L_K(E)$ for each $n\geq 1$ (where $\langle E^0 \setminus M_1 \rangle = 0$). However, $$\bigcup_{i=1}^\infty\langle E^0 \setminus M_i \rangle = \langle \{v_i \mid i \geq 1\}\rangle$$ is not a prime ideal, by Theorem~\ref{primeclass}, since $E^0\setminus \{v_i \mid i \geq 1\} = \{u_1, u_2\}$ does not satisfy MT3. (Note that $E^0 \cap \langle \{v_i \mid i \geq 1\}\rangle = \{v_i \mid i \geq 1\}$ by Remark~\ref{HcupstuffcapE0equalsH}.)  \hfill $\Box$
\end{example}

We mention in passing that it can be shown that adjoining a unit to the ring $L_K(E)$ given in the previous example produces a unital ring with a chain of prime ideals whose union is not prime.

\begin{corollary} \label{primecor}
Let $K$ be a field, and let $E$ be an acyclic graph such that $B_H = \emptyset$ for every hereditary saturated $H \subseteq E^0$. Then an ideal $I$ of $L_K(E)$ is prime if and only if it is of the form $I=\langle H \rangle$, where $H$ is a proper subset of $E^0$ such that $E^0\setminus H$ is a maximal tail.
\end{corollary}

\begin{proof}
If $I$ is a prime ideal, then, by Theorem~\ref{primeclass}, $I = \langle I \cap E^0 \rangle$, since $E$ is acyclic and $B_{I \cap E^0} = \emptyset$, where $E^0 \setminus (I \cap E^0)$ satisfies MT3. Since $I$ is an ideal, $I \cap E^0$ must be hereditary and saturated, Remark~\ref{HcupstuffcapE0equalsH}. From this it follows that $E^0 \setminus (I \cap E^0)$ satisfies MT1 and MT2, and is hence a maximal tail. Finally, $I \cap E^0 \neq E^0$, since $\langle E^0 \rangle = L_K(E)$, and $I$ is prime.

Conversely, suppose that $H \subset E^0$ is such that $E^0\setminus H$ satisfies MT1, MT2, and MT3. Then, as mentioned above, $H$ must be hereditary and saturated. Hence, by hypothesis, $B_H = \emptyset$. Moreover, by Remark~\ref{HcupstuffcapE0equalsH}, $\langle H \rangle$ must be a proper ideal of $L_K(E)$. Thus $I = \langle H \rangle$ is prime, by Theorem~\ref{primeclass}.
\end{proof}

\begin{example}\label{ExampleSpecnotDC}
Let $K$ be any field, and let $E$ be the graph pictured here.
$$
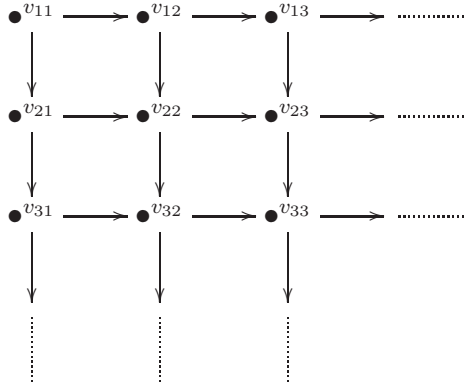
\xymatrix{ 
{\bullet}^{v_{11}} \ar [r]  \ar[d] & {\bullet}^{v_{12}} \ar [r] \ar[d] & {\bullet}^{v_{13}} \ar [r]  \ar[d] & \ar@{.}[r] & \\
{\bullet}^{v_{21}} \ar [r] \ar[d] & {\bullet}^{v_{22}} \ar [r] \ar[d] & {\bullet}^{v_{23}} \ar [r] \ar[d] & \ar@{.}[r] &\\
{\bullet}^{v_{31}} \ar [r] \ar[d] & {\bullet}^{v_{32}} \ar [r] \ar[d] & {\bullet}^{v_{33}} \ar [r] \ar[d] & \ar@{.}[r] &\\
\ar@{.}[d] & \ar@{.}[d] & \ar@{.}[d] & & \\
&&&&
}$$
Then the following are all the maximal tails of $E$: 
$$M_n = \{v_{ij} \mid i \leq n\} \  (n \in \Z^+),  \ \ \ N_n = \{v_{ij} \mid j \leq n\} \ (n \in \Z^+),  \ \  \mbox{ and }  \ E^0.$$
That is, $M_n$ consists of the vertices in the first $n$ rows of $E$, and $N_n$ consists of the vertices in the first $n$ columns of $E$.   Clearly $E$ is acyclic, and $B_H = \emptyset$ for every hereditary saturated set $H \subseteq E^0$ (since $E$ has only regular vertices). Thus, by Corollary~\ref{primecor}, the prime spectrum of $L_K(E)$ has the following Hasse diagram.
$$
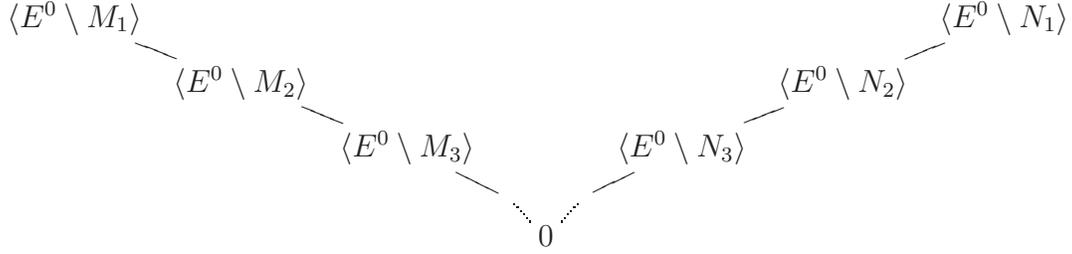
\xymatrix@=.5pc{ 
\langle E^0 \setminus M_1 \rangle \ar@{-}[dr] & & & & & & & & \langle E^0 \setminus N_1 \rangle \ar@{-}[dl] \\
& \langle E^0 \setminus M_2 \rangle \ar@{-}[dr] & & & & & & \langle E^0 \setminus N_2 \rangle \ar@{-}[dl] & \\
& & \langle E^0 \setminus M_3 \rangle \ar@{-}[dr] & & & & \langle E^0 \setminus N_3 \rangle \ar@{-}[dl] & & \\
& & & \ar@{.}[dr] & & \ar@{.}[dl] & & & \\
& & & & 0 & & & &
}$$

\noindent
We will revisit this example in Section~\ref{AandRsection}.  \hfill $\Box$
\end{example}

The following  lemma will be useful in the sequel.

\begin{lemma} \label{taillemma}
Let $E$ be a graph, $\, \{M_i \mid i \in I\}$ a collection of maximal tails of $E^0$, and $M = \bigcup_{i\in I} M_i$. Then the following hold.
\begin{enumerate}
\item[$(1)$] $M$ satisfies \emph{MT1} and \emph{MT2}.
\item[$(2)$] $M$ satisfies \emph{MT3} if and only if for all $u, v \in M$ there exists $i \in I$ such that $u, v \in M_i$.
\end{enumerate}
\end{lemma}

\begin{proof}
(1) Suppose that $v\in M$ and $u\in E^0$ are such that $u \geq v$. Then there is some $i \in I$ such that $v \in M_i$. Since $M_i$ satisfies MT1, we have $u \in M_i$, and hence $u \in M$, showing that $M$ satisfies MT1. 

Next, let $v\in M$ be a regular vertex, and let $i \in I$ be such that $v \in M_i$. Since $M_i$ satisfies MT2, there exists $e \in E^1$ such that $\so(e)=v$ and $\ra(e) \in M_i$. Hence $\ra(e) \in M$, and therefore $M$ satisfies MT2.

(2) Suppose that for all $u, v \in M$ there exists $i \in I$ such that $u, v \in M_i$. Let $u, v \in M$ and let $i \in I$ be such that $u, v \in M_i$. Since $M_i$ satisfies MT3, there exists $w \in M_i$ such that $u \geq w$ and $v \geq w$. Since $w \in M$, it follows that $M$ satisfies MT3.

Conversely, suppose that $M$ satisfies MT3, and let $u, v \in M$. Then there exists $w \in M$ such that $u \geq w$ and $v \geq w$. Let $i \in I$ be such that $w \in M_i$. Then $u, v \in M_i$, since $M_i$ satisfies MT1.
\end{proof}

\section{The graph $E_P$ and the prime spectrum of $L_K(E_P)$}

In the previous two sections we have developed all the necessary ideas to put us in position to present the key construction, in which we build a directed graph  $E_P$ from any poset $P$.  In the sequel $P$ will always be assumed to be nonempty.

\begin{definition} \label{mainconstr}
Given a partially ordered set $\, (P, \leq)$ we define the graph $E_P$ as follows:
$$E^0_P = \{v_p \mid p \in P\} \ \ \ \ \mbox{ and }  \ \ \ \ \ E^1_P = \{e_{p,q}^i \mid i \in \N,  \text{ and } p,q \in P \text{ satisfy } p>q\},$$ where $\, \so(e_{p,q}^i) = v_p$ and $\, \ra(e_{p,q}^i) = v_q$ for all $i \in \N$.
\hfill $\Box$
\end{definition}

Less formally, $E_P$ is built from $P$ by viewing the elements of $P$ as vertices, and putting countably infinitely many edges from $v_p$ to $v_q$ whenever $p > q$ in $P$.    

\begin{example}\label{EPExample}
Let $P$ be the poset with the following Hasse diagram.
$$
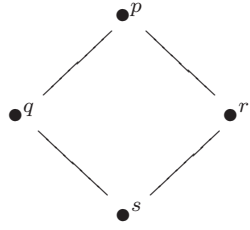
\xymatrix{ 
    & {\bullet}^{p} \ar@{-}[dr]   &    \\
{\bullet}^q \ar@{-}[ur]  &  & {\bullet}^r \ar@{-}[dl]  \\
  & {\bullet}^s \ar@{-}[ul]  &  \\
}$$
\noindent
Then $E_P$ is the graph below.
$$
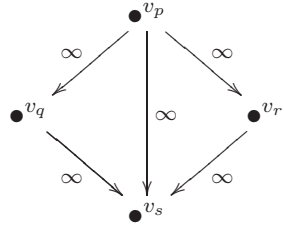
\xymatrix{ 
    & {\bullet}^{v_p} \ar[dl]_\infty  \ar[dr]^\infty  \ar[dd]^\infty  &    \\
{\bullet}^{v_q} \ar[dr]_\infty  &  & {\bullet}^{v_r} \ar[dl]^\infty  \\
  & {\bullet}^{v_s} \ar@{-}[ul]  &  \\
}$$

\vspace{-.25in}
\hfill $\Box$
\end{example}

\medskip

By inserting infinitely many edges between all connected vertices, the graph  $E_P$ and  the  poset $\Spec(L_K(E_P))$  are endowed with particularly nice properties, as the next lemma shows.   (As a reminder, for any graph $F$ and $v\in F^0$, $\MT(v)$ denotes the set $\{w\in F^0  \ | \ w\geq v\}$.)  

\begin{lemma} \label{chrislemma}
Let $\, (P, \leq)$ be a partially ordered set.   
\begin{enumerate}
\item[$(1)$] For all $p,q \in P$ we have $p > q$ if and only if $v_p > v_q$ if and only if $\, \MT(v_p) \subset \MT(v_q)$.
\item[$(2)$] For all $p,q \in P$ we have $p \geq q$ if and only if $v_p \geq v_q$ if and only if $\, \MT(v_p) \subseteq \MT(v_q)$.
\item[$(3)$] $E_P$ is acyclic.
\item[$(4)$] $(E^0_P, \leq)$ is partially ordered and is order-isomorphic to $\, (P, \leq)$. 
\item[$(5)$] An ideal $I$ of $L_K(E_P)$ is prime if and only if it is of the form $I =\langle H \rangle$, where $H$ is a proper subset of $E_P^0$ such that $E_P^0\setminus H$ satisfies \emph{MT1} and \emph{MT3}.
\end{enumerate}
\end{lemma}

\begin{proof}
(1) Let $p,q \in P$ be distinct elements. If $p > q$, then $\so(e_{p,q}^1) = v_p$ and $\ra(e_{p,q}^1) = v_q$, which implies that $v_p > v_q$. Conversely, if $v_p > v_q$, then there is a path $\alpha \in \pth(E)$ of the form $\alpha  = e_{p,a_1}^1e_{a_1,a_2}^1 \dots e_{a_n,q}^1$ ($a_1, \dots, a_n \in P$). It follows from the definition of $E_P$ that $p > a_1 > a_2 > \dots > a_n > q$, and hence that $p > q$, since $>$ is transitive.

Next, if $v_p > v_q$ and $w \in E^0_P$ is such that $w \geq v_p$, then clearly $w \geq v_q$. It follows that $\MT(v_p) \subseteq \MT(v_q)$. Moreover, it cannot be the case that $v_q \geq v_p$, since then we would have $p>q$ and $q \geq p$, which is not possible in a partially ordered set. Hence $v_q \notin \MT(v_p)$, and therefore $\MT(v_p) \subset \MT(v_q)$. Conversely, if $\MT(v_p) \subset \MT(v_q)$, then, in particular, $v_p \in \MT(v_q)$, and hence $v_p > v_q$, by the definition of $\MT(v_q)$.

(2) This follows immediately from (1).

(3) By the definition of $E_P$, there are no loops in $\pth(E_P)$. Thus for $E_P$ to be not acyclic there must be distinct vertices $v_p, v_q \in E^0_P$ ($p,q \in P$) such that $v_p > v_q$ and $v_q > v_p$. But, by (1) this can happen only if $p > q$ and $q > p$, which is not possible, since $\geq$ is antisymmetric. Hence $E_P$ must be acyclic.

(4) As mentioned above, $(E^0_P, \leq)$ is preordered. Hence, by (3), $(E^0_P, \leq)$ is partially ordered, and by (2), $(E^0_P, \leq) \cong (P, \leq)$.

(5) Let $v_p, v_q \in E^0_P$ be two vertices ($p,q \in P$). If $v_p > v_q$, then $|\so^{-1}(v_p) \cap \ra^{-1}(v_q)| = \aleph_0$, by (1) and the definition of $E_P$. Hence, for any $H \subseteq E^0_P$ either $|\so^{-1}(v_p) \cap \ra^{-1}(E^0_P \setminus H)| = 0$ or  $|\so^{-1}(v_p) \cap \ra^{-1}(E^0_P \setminus H)| \geq \aleph_0$. In particular, if $H \subseteq E^0_P$ is hereditary and saturated, then $B_H = \emptyset$.

Finally, note that every subset of $E^0_P$ satisfies MT2 vacuously, since $E^0_P$ contains no regular vertices. The desired conclusion now follows from Corollary~\ref{primecor} and (3).
\end{proof}

We  now  relate the structure of any partially ordered set $(P, \leq)$ to that of $\Spec(L_K(E_P))$.

\begin{proposition} \label{dccthrm}
Let $\, (P, \leq)$ be a partially ordered set and $K$ a field. Define the function 
$$\phi : P \to \Spec(L_K(E_P)) \ \mbox{ by setting } \ \phi(p) = \langle E^0_P \setminus \MT(v_p) \rangle \mbox{ for each } p \in P.$$
Then $\phi$ is an order-embedding. Moreover, $\phi$ is an order-isomorphism if and only if $\, (P, \leq)$ satisfies \emph{DCC}. 
\end{proposition}

\begin{proof}
First note that for any vertex $u$ (in any graph), $\MT(u)$ is nonempty and satisfies MT1 and MT3. Hence, by Lemma~\ref{chrislemma}(5), $\langle E^0_P \setminus \MT(v_p) \rangle \in \Spec(L_K(E_P))$ for all $p \in P$. 

Now let $p,q \in P$ and suppose that $q \leq p$. Then $\MT(v_p) \subseteq \MT(v_q)$, by Lemma~\ref{chrislemma}(2), and hence $\langle E^0_P \setminus \MT(v_q) \rangle \subseteq \langle E^0_P \setminus \MT(v_p) \rangle$. That is, $\phi(q) \subseteq \phi(p)$. Hence $\phi$ is order-preserving.

Next, suppose that $\phi(q) \subseteq \phi(p)$. Then $\langle E^0_P \setminus \MT(v_q) \rangle \subseteq \langle E^0_P \setminus \MT(v_p) \rangle$. By Remark~\ref{HcupstuffcapE0equalsH}, we have $$E^0_P \setminus \MT(v_q) = \langle E^0_P \setminus \MT(v_q) \rangle \cap E^0_P \subseteq \langle E^0_P \setminus \MT(v_p) \rangle \cap E^0_P = E^0_P \setminus \MT(v_p),$$ and hence $\MT(v_p)\subseteq \MT(v_q)$. By Lemma~\ref{chrislemma}(2), this implies that $q \leq p$, and hence $\phi$ is order-reflecting, and therefore an order-embedding.

\smallskip

To prove the final claim, suppose that $(P, \leq)$ satisfies DCC. To conclude that $\phi$ is a bijection (and hence an order-isomorphism) it suffices to show that if $I$ is a prime ideal of $L_K(E_P)$, then $I =  \langle E^0_P \setminus \MT(u) \rangle$ for some $u \in E^0_P$. So let $I$ be a prime ideal of $L_K(E_P)$. By Lemma~\ref{chrislemma}(5), $I =\langle H \rangle$, where $H$ is a proper subset of $E_P^0$ such that $E_P^0\setminus H$ satisfies MT1 and MT3. Let $Q = \{p \mid v_p \in E_P^0\setminus H\} \subseteq P$. Since $P$ satisfies DCC, so does $Q$, and hence $Q$ must have a minimal element. Suppose that $p,q \in Q$ are two minimal elements. Since $v_p, v_q \in E_P^0\setminus H$ and this set satisfies MT3, there must exist some $r \in Q$ such that $v_r \in E_P^0\setminus H$, $v_p \geq v_r$, and $v_q \geq v_r$. By Lemma~\ref{chrislemma}(2), this implies that $p \geq r$ and $q \geq r$. But since $p$ and $q$ are minimal, we conclude that $p=r=q$, and hence $Q$ must have a unique minimal element $r$. Now since $E_P^0\setminus H$ satisfies MT3, it follows that for any $v_p \in E_P^0\setminus H$ we have $v_p \geq v_r$, and hence $E_P^0\setminus H \subseteq \MT(v_r)$. Also, if $v_p \in E^0$ is any vertex such that $v_p \geq v_r$, then $v_p \in E_P^0\setminus H$, since this set satisfies MT1. Therefore $E_P^0\setminus H = \MT(v_r)$, and hence $H = E_P^0\setminus \MT(v_r)$, as desired.

Conversely, suppose that $\phi$ is an order-isomorphism. Then every prime ideal of $L_K(E_P)$ must be of the form $\langle E^0_P \setminus \MT(u) \rangle$ for some $u \in E^0_P$. Seeking a contradiction, suppose that $(P,\leq)$ does not satisfy DCC. Then there is a subset $\{p_i \mid i \in \N\}$ of $P$ satisfying $p_0 > p_1 > p_2 > \dots$. By Lemma~\ref{chrislemma}(1), $$\MT(v_{p_0}) \subset \MT(v_{p_1}) \subset \MT(v_{p_2}) \subset \dots,$$ and by Lemma~\ref{taillemma}, $M = \bigcup_{i \in \N} \MT(v_{p_i})$ is a maximal tail, since each $\MT(v_{p_i})$ is. Hence, by Lemma~\ref{chrislemma}(5), $\langle E^0_P \setminus M \rangle$ is a prime ideal of $L_K(E_P)$. However, $M$ cannot be of the form $M=\MT(u)$ for any $u \in E^0_P$, since otherwise $u \in \MT(v_{p_i})$ for some $i \in \N$, and then we would have $M = \MT(v_{p_i})$. This gives the desired contradiction, and hence $(P,\leq)$ must satisfy DCC. 
\end{proof}

Our next goal is to describe $\Spec(L_K(E_P))$ for an arbitrary partially ordered set $(P, \leq)$. Towards that end we shall require the following notation.

\begin{definition} \label{precdef}
Let $\, (P, \leq)$ be a partially ordered set. Define a binary relation $\, \preceq$ on the power set $\mathcal{P}(P)$ of $P$ as follows. Given $S_1, S_2 \in \mathcal{P}(P) \setminus \{\emptyset\}$, write $S_1 \preceq S_2$ if for every $s_2 \in S_2$ there exists $s_1 \in S_1$ such that $s_1 \leq s_2$. If $S_1 \preceq S_2$ and $S_2 \preceq S_1$, then we write $S_1 \approx S_2$.  \hfill $\Box$ 
\end{definition}

It is routine to verify that $(\mathcal{P}(P), \preceq)$ is a preordered set,  that $\approx$ is an equivalence relation on $\mathcal{P}(P)$, and that  $\preceq$ induces a partial order on the set $\mathcal{P}(P)/\approx$ of $\approx$-equivalence classes in $\mathcal{P}(P)$. Given $S \in \mathcal{P}(P)$ we shall denote the $\approx$-equivalence class of $S$ by $[S]$.

\begin{definition} \label{leqAdef}
Given a partially ordered set $\, (P, \leq)$, let $$\AT(P) = P \cup \{x_{[S]} \mid S \subseteq P \text{ is downward directed with no least element}\}.$$ Further, we extend $\, \leq$ to a binary relation $\, \leq_\AT$ on $\AT(P)$, as follows.  For all $p,q \in \AT(P)$ let $p \leq_\AT q$ if one of the following holds:
\begin{enumerate}
\item[$(1)$] $p,q \in P$ and $p \leq q$;
\item[$(2)$] $p \in P$, $q = x_{[S]} \in \AT(P)\setminus P$, and $p \leq s$ for all $s \in S$;
\item[$(3)$] $p = x_{[S]} \in \AT(P)\setminus P$, $q \in P$, and $s \leq q$ for some $s \in S$;
\item[$(4)$] $p,q \in \AT(P)\setminus P$, $p = x_{[S]}$, $q = x_{[T]}$, and $S \preceq T$.  \hfill $\Box$
\end{enumerate}
\end{definition}

\medskip

Loosely speaking, to construct $\AT(P)$ from $P$ we {\it adjoin} a greatest lower bound to each downward directed subset $S$ of $P$ containing no least element.   (We note that $S$ might already have a greatest lower bound in $P$. Even so, we still adjoin a ``new" greatest lower bound for $S$ in $\AT(P)$.)     

\begin{example}\label{A(P)Example}
Let $P = \{0\} \cup \{ \frac{1}{n} \ | \ n \in \Z^+ \}$ with the usual order.   Clearly  $S = \{ \frac{1}{n} \ | \ \in \Z^+ \} $ is a subset of $P$ which contains no least element.   Clearly also any infinite subset $S^\prime$ of $S$ has $[S^\prime] = [S]$, while any finite subset of $S$ contains a least element. So $$\AT(P) =  P \cup \{x_{[S]}\} ,$$ where $0 <_\AT x_{[S]} <_\AT \frac{1}{n}$ for all $n \in \Z^+$.  
\hfill $\Box$
\end{example}

Now invoking Remark~\ref{dccvsglb}, we easily get the following.

\begin{lemma}\label{dcciffA(P)equalsP}
Let $P$ be any partially ordered set.  Then $P$ has \emph{DCC} if and only if  $\AT(P) = P$. 
\end{lemma}

In particular, we note that the operation $\AT$ on posets should not be viewed as a closure operation, i.e., $\AT(\AT(P))$ need not equal $\AT(P)$.  (This is  readily seen by considering the poset $P$ of Example~\ref{A(P)Example}.)  

\medskip

Although we have already defined the binary relation $\leq_\AT$ on $\AT(P)$ for any poset $P$, we have yet to make any claims regarding  relational properties of $\leq_\AT$. While the reader's likely hunch that $\leq_\AT$ is a partial order on $\AT(P)$  is indeed true, a direct first-principles verification of this fact  requires  some significant effort (including checking numerous cases). As it turns out, this fact will follow as a nice byproduct of the following theorem, which is  the main result of this article.

\begin{theorem} \label{EPspec}
Let $\, (P, \leq)$ be a partially ordered set and $K$ a field. Then $\, (\AT(P),  \leq_\AT)$ is a partially ordered set, and $$\, (\Spec(L_K(E_P)), \subseteq) \ \cong \  (\AT(P), \leq_\AT).$$
\end{theorem}

\begin{proof}
Let $\phi : P \to \Spec(L_K(E_P))$ be as in Proposition~\ref{dccthrm}. The goal of Claims 1 and 2 below is to extend $\phi$ to a function $\varphi: \AT(P) \to \Spec(L_K(E_P))$. Then in Claims 3 and 4 we establish that $\varphi$ is a bijection that respects the binary relations $\leq_\AT$ and $\subseteq$. Both assertions of the theorem will then follow.

\smallskip

\underbar{Claim 1}: If $I \in \Spec(L_K(E_P)) \setminus \phi(P)$, then there exists a downward directed subset $S \subseteq P$ with no least element such that $I = \bigcap_{s\in S} \langle E^0_P \setminus \MT(v_s) \rangle$.   

To see this, we note that, by Lemma~\ref{chrislemma}(5), any such ideal $I$ must be of the form $I =\langle H \rangle$, where $H \subset E_P^0$ is such that $M = E_P^0\setminus H$ satisfies MT1 and MT3. Let $S = \{s \mid v_s \in M\} \subseteq P$. Then $M = \bigcup_{s \in S} \MT(v_s)$, since $M$ satisfies MT1. By Lemma~\ref{chrislemma}(4), $(M, \leq)$ is a partially ordered set (since $M \subseteq E^0_P$) and $(M, \leq) \cong (S, \leq)$. Since $M$ satisfies MT3, $(M, \leq)$ is downward directed, and hence so is $(S, \leq)$. Moreover, $S$ cannot have a least element $s \in S$, since otherwise we would have $M = \MT(v_s)$, and therefore $I = \langle E^0_P \setminus \MT(v_s) \rangle = \phi(v_s)$, contrary to the choice of $I$. Now, $$I = \langle E^0_P \setminus M \rangle = \Big\langle E^0_P \setminus \bigcup_{s \in S} \MT(v_s) \Big\rangle = \Big\langle \bigcap_{s\in S} E^0_P \setminus \MT(v_s) \Big\rangle,$$ and hence we seek to show that the last expression in the previous display is equal to $\bigcap_{s\in S} \langle E^0_P \setminus \MT(v_s) \rangle$.

Let $x \in \langle \bigcap_{s\in S} E^0_P \setminus \MT(v_s) \rangle$ be any element. Then $x \in \langle E^0_P \setminus \MT(v_s) \rangle$ for all $s \in S$, and hence $x \in \bigcap_{s\in S} \langle E^0_P \setminus \MT(v_s) \rangle$, showing that $$\Big\langle \bigcap_{s\in S} E^0_P \setminus \MT(v_s) \Big\rangle \subseteq \bigcap_{s\in S} \langle E^0_P \setminus \MT(v_s) \rangle.$$ For the opposite inclusion, note that since $S$ is downward directed, $\{\langle E^0_P \setminus \MT(v_s) \rangle \mid s \in S\}$ is a downward directed subset of $\Spec(L_K(E))$, by  Lemma~\ref{chrislemma}(2), and hence $\bigcap_{s\in S} \langle E^0_P \setminus \MT(v_s) \rangle$ is a prime ideal, by Proposition~\ref{primeintersect}. It follows, by Lemma~\ref{chrislemma}(5), that $\bigcap_{s\in S} \langle E^0_P \setminus \MT(v_s) \rangle = \langle H' \rangle$ for some $H' \subset E_P^0$. Let $v \in H'$ be any vertex. Then for all $s \in S$ we have $v \in \langle E^0_P \setminus \MT(v_s) \rangle$, and hence $v \in E^0_P \setminus \MT(v_s)$, by Remark~\ref{HcupstuffcapE0equalsH}. Thus $v \in \bigcap_{s\in S} E^0_P \setminus \MT(v_s) $, from which we conclude that $$\bigcap_{s\in S} \langle E^0_P \setminus \MT(v_s) \rangle = \langle H' \rangle \subseteq \Big\langle \bigcap_{s\in S} E^0_P \setminus \MT(v_s) \Big\rangle,$$ and therefore $$\bigcap_{s\in S} \langle E^0_P \setminus \MT(v_s) \rangle = \Big\langle \bigcap_{s\in S} E^0_P \setminus \MT(v_s) \Big\rangle.$$
This establishes Claim 1.  

For any downward directed subset $S \subseteq P$ with no least element, define $$I_{[S]} = \bigcap_{s\in S} \langle E^0_P \setminus \MT(v_s) \rangle.$$

\underbar{Claim 2}:  If $S,S' \subseteq P$ are downward directed subsets with no least elements, such that $S \approx S'$, then $I_{[S]} = I_{[S']}.$  

Noting that if $S \approx S'$, then $\bigcup_{s\in S}\MT(v_s) = \bigcup_{s\in S'}\MT(v_s)$, by Lemma~\ref{chrislemma}(2), this claim follows by using the final displayed conclusion of Claim 1, which yields 
$$I_{[S]} = \bigcap_{s\in S} \langle E^0_P \setminus \MT(v_s) \rangle = \Big\langle \bigcap_{s\in S} E^0_P \setminus \MT(v_s) \Big\rangle = \Big\langle \bigcap_{s\in S'} E^0_P \setminus \MT(v_s) \Big\rangle = \bigcap_{s\in S'} \langle E^0_P \setminus \MT(v_s) \rangle = I_{[S']} .$$ 

We now define  the function $\varphi : \AT(P) \to \Spec(L_K(E_P))$ by setting 
$$\varphi (p)=\left\{
\begin{array}{ll}
\phi(p) & \mbox{if } p \in P\\
I_{[S]} & \mbox{if } p = x_{[S]} \in \AT(P)\setminus P
\end{array}\right..$$
By Claim 2, $\varphi$ is well-defined, and by Claim 1, $\varphi$ is surjective.   

\smallskip

\underbar{Claim 3}: $\varphi$ is a bijection. 

To establish this claim, we recall that by Proposition~\ref{dccthrm}, $\phi: P \to \phi(P)$ is a bijection, and hence it is enough to show that if $S,S' \subseteq P$ are downward directed subsets with no least elements such that $S \not\approx S'$, then $I_{[S]} \neq I_{[S']}$. For this purpose suppose, without loss of generality, that $S \not\preceq S'$, that is there exists $s' \in S'$ such that $s \not\leq s'$ for all $s \in S$. Then by Lemma~\ref{chrislemma}(2), $\MT(v_{s'}) \not\subseteq \MT(v_s)$ for all $s \in S$, and hence $v_{s'}\notin \MT(v_s)$ for all $s \in S$. It follows that $\bigcup_{s\in S'} \MT(v_s) \not\subseteq \bigcup_{s\in S} \MT(v_s)$, and therefore, by Remark~\ref{HcupstuffcapE0equalsH}, $$I_{[S]} = \Big\langle \bigcap_{s\in S} E^0_P \setminus \MT(v_s) \Big\rangle \neq \Big\langle \bigcap_{s\in S'} E^0_P \setminus \MT(v_s)\Big\rangle = I_{[S']},$$ as desired.

\smallskip

\underbar{Claim 4}:  $p \leq_\AT q$ if and only if $\varphi (p) \subseteq \varphi (q)$ for any pair $p,q \in \AT(P)$.

With Definition~\ref{leqAdef} in mind, there are four cases to check. First, by Proposition~\ref{dccthrm}, $\phi$ is an order-embedding, so whenever $p,q \in P$ we have $p \leq_\AT q$ if and only if  $p \leq q$ if and only if $\varphi (p) \subseteq \varphi (q)$.  Second,   suppose that $p \in P$ and $q = x_{[S]} \in \AT(P)\setminus P$. Then 
\begin{eqnarray*}
p \leq_\AT q  & \Leftrightarrow  & p \leq s \text{ for all } s \in S \text{ (by Definition~\ref{leqAdef})} \\
& \Leftrightarrow & \phi (p) \subseteq \phi (s) \text{ for all } s \in S \text{ (by Proposition~\ref{dccthrm})} \\
& \Leftrightarrow & \langle E^0_P \setminus \MT(v_p) \rangle \subseteq \langle E^0_P \setminus \MT(v_s) \rangle \text{ for all } s \in S \\
& \Leftrightarrow & \langle E^0_P \setminus \MT(v_p) \rangle \subseteq I_{[S]} \\
& \Leftrightarrow & \varphi (p) \subseteq \varphi (q).
\end{eqnarray*}
\noindent
Third,  suppose that $p = x_{[S]} \in \AT(P)\setminus P$ and $q \in P$. Note that if $$I_{[S]} = \Big\langle \bigcap_{s\in S} E^0_P \setminus \MT(v_s) \Big\rangle \subseteq \langle E^0_P \setminus \MT(v_q) \rangle,$$ then $\bigcap_{s\in S} E^0_P \setminus \MT(v_s) \subseteq  E^0_P \setminus \MT(v_q)$, by Remark \ref{HcupstuffcapE0equalsH}, and hence $\bigcup_{s\in S} \MT(v_s) \supseteq \MT(v_q)$. Thus $v_q \in \MT(v_s)$ for some $s \in S$, and therefore $\MT(v_q) \subseteq \MT(v_s)$. It follows that $I_{[S]} \subseteq \langle E^0_P \setminus \MT(v_q) \rangle$ if and only if $\langle E^0_P \setminus \MT(v_s) \rangle \subseteq \langle E^0_P \setminus \MT(v_q) \rangle$ for some $s \in S$. Therefore
\begin{eqnarray*}
p \leq_\AT q & \Leftrightarrow & s \leq q \text{ for some } s \in S \text{ (by Definition~\ref{leqAdef})}\\
& \Leftrightarrow & \phi (s) \subseteq \phi (q) \text{ for some } s \in S \text{ (by Proposition~\ref{dccthrm})}\\
&\Leftrightarrow & \langle E^0_P \setminus \MT(v_s) \rangle \subseteq \langle E^0_P \setminus \MT(v_q) \rangle \text{ for some } s \in S\\
&\Leftrightarrow & I_{[S]} \subseteq \langle E^0_P \setminus \MT(v_q) \rangle\\
& \Leftrightarrow & \varphi (p) \subseteq \varphi (q).
\end{eqnarray*}
\noindent
Fourth and finally, suppose that $p,q \in \AT(P)\setminus P$, $p = x_{[S]}$, and $q = x_{[T]}$. Then 
\begin{eqnarray*}
p \leq_\AT q & \Leftrightarrow & \text{ for all } t \in T \text{ there exists } s \in S \text{ such that }
s \leq t \text{ (by Definition~\ref{leqAdef})}\\
& \Leftrightarrow & \text{ for all } t \in T \text{ there exists } s \in S \text{ such that }
\phi(s) \subseteq \phi(t) \text{ (by Proposition~\ref{dccthrm})}\\
& \Leftrightarrow & \text{ for all } t \in T \text{ there exists } s \in S \text{ such that }
\langle E^0_P \setminus \MT(v_s) \rangle \subseteq \langle E^0_P \setminus \MT(v_t) \rangle\\
& \Leftrightarrow & I_{[S]} \subseteq \langle E^0_P \setminus \MT(v_t) \rangle \text{ for all } t \in T\\
&\Leftrightarrow & I_{[S]} \subseteq I_{[T]}\\
&\Leftrightarrow & \varphi (p) \subseteq \varphi (q).
\end{eqnarray*}
\noindent
Thus $p \leq_\AT q$ if and only if $\varphi (p) \subseteq \varphi (q)$ for all $p,q \in \AT(P)$, which establishes Claim 4. 

From Claims 3 and 4 we conclude that $(\AT(P), \leq_\AT)$ is a partially ordered set, since $(\Spec(L_K(E_P)), \subseteq)$ is, and that $(\Spec(L_K(E_P)), \subseteq) \cong (\AT(P),  \leq_\AT)$.
\end{proof}

Using the fact that $(\AT(P), \leq_\AT)$ contains a copy of $(P, \leq)$, or by invoking Proposition~\ref{dccthrm}, we get the following. 

\begin{corollary}\label{PembedsinSpec}
For any partially ordered set $\, (P, \leq)$ there is an order-embedding of $\, (P, \leq)$ into $\, (\Spec(L_K(E_P)), \subseteq)$.
\end{corollary}



We close this section by noting that, by Lemma~\ref{dcciffA(P)equalsP},  Theorem~\ref{EPspec} is the generalization of the final statement of Proposition~\ref{dccthrm} to all partially ordered sets.

\section{The posets $\AT(P)$ and $\RT(P)$}\label{AandRsection}

In this section we identify, up to poset isomorphism, those partially ordered sets which can be expressed in the form $(\AT(P), \leq_\AT)$. 

\begin{definition} \label{conddef}
Let $\, (P, \leq)$ be a partially ordered set. For any $S \subseteq P$ and $p \in P$, we write $p=\mathrm{glb}(S)$ in case $S$ has a greatest lower bound in $P$, and it is equal to $p$. Also, let   
$$\RT(P) = \{ p \in P \mid p \neq \mathrm{glb}(S) \text{ for all } S \subseteq P  \text{ downward directed without least element}\}.$$
We view $\RT(P)$ as a partially ordered subset of $P$.  
\hfill $\Box$
\end{definition}

Rephrased, $\RT(P)$ is constructed from $P$ by \textit{removing} each element of $P$ which arises as the greatest lower bound of a downward directed subset of $P$ not containing a least element. As we shall see in Proposition~\ref{R(A(P))equalsP}, the operation $\RT$ on posets ``undoes" the effects of applying $\AT$. Understanding precisely how these two operations interact with each other will be key to classifying the posets of the form $(\AT(P), \leq_\AT)$.
 
\begin{example}\label{R(P)Example}
Let $P$ be the poset $\{0\} \cup \{ \frac{1}{n} \ | \ n \in \Z^+ \}$ of Example \ref{A(P)Example}. Then $\RT(P) = S =  \{ \frac{1}{n} \ | \ n \in \Z^+ \}$.  This is because $0 = \mathrm{glb}(S)$ and $S$ has no least element, so $0 \not\in \RT(P)$.  On the other hand, each element $\frac{1}{n}$ is in $\RT(P)$, because if $\frac{1}{n}$ is the greatest lower bound for some subset $Q$ of $P$, then $Q$ must be finite, and hence $\frac{1}{n}\in Q$. This also shows that $\RT(S) = S$.

Observing that $P \cong \AT(S)$, the above computation implies that $\RT(\AT(S)) = S.$ Arguing in a similar manner, we see also that $\RT(\AT(P)) = P$. (In this case $x_{[S]} \not\in \RT(\AT(P))$ because $x_{[S]} = \mathrm{glb}(S)$, while $0\in \RT(\AT(P))$ because $0$ is \textit{not} the greatest lower bound of $S$ in $\AT(P)$.)

Now applying the two operations in the opposite order, we have $\AT(\RT(S)) = \AT(S) \cong P \not\cong S$, whereas $\AT(\RT(P)) = \AT(S) \cong P$.
\hfill $\Box$
\end{example}

We require a bit more notation and several preliminary results in order to classify the partially ordered sets that can be expressed in the form $(\AT(P), \leq_\AT)$.

\begin{notation}\label{GLBDCDDnotation}
We assign the indicated names in case the poset $\, (P, \leq)$ satisfies the following germane properties.  

\medskip

(GLB) \ \  Every downward directed subset of $P$ has a greatest lower bound in $P$.

\medskip

 (DC) \ \ \ \   For every downward directed subset $S$ of $P$ and every $p \in \RT(P)$ satisfying 
 
 \qquad \qquad $p \geq \mathrm{glb}(S)$, we have $p \geq s$ for some $s \in S$.

\medskip

 (DD) \ \ \ \   For every downward directed subset $S$ of $P$ such that $\mathrm{glb}(S) \in P$, there exists 
 
 \qquad \qquad a downward directed subset $T$ of $\RT(P)$ satisfying $\mathrm{glb}(S) = \mathrm{glb}(T)$.  

\medskip
\noindent
 Of course GLB stands for \textit{greatest lower bound}. We have chosen the name DC to stand for \textit{directed compatibility}, while DD connotes \textit{directed discreteness}.  \hfill $\Box$ 
\end{notation}

Our reason for describing the condition in DD as ``discreteness" will be clarified in Proposition~\ref{KAPiffDD}. Also, in Theorem~\ref{bergman} we will show that for a poset $P$, condition GLB is equivalent to the condition that every totally ordered subset of $P$ has a greatest lower bound in $P$.

\begin{lemma} \label{largedd}
Let $\, (P, \leq)$ be a partially ordered set, let $S$ be a downward directed subset of $P$ such that $\, \mathrm{glb}(S) \in P$, and let $$S' = S \cup \{t \in P \mid \exists s \in S \text{ such that } s \leq t\}.$$ Then the following hold.
\begin{enumerate}
\item[$(1)$] The set $S'$ is downward directed, and $\, \mathrm{glb}(S) = \mathrm{glb}(S')$.
\item[$(2)$] If $S$ has no least element, then neither does $S'$, and for all $s \in S$ there exists $t \in S$ such that $t < s$.
\end{enumerate}
\end{lemma}

\begin{proof}
(1) Let $t_1, t_2 \in S'$ be any elements, and let $s_1, s_2 \in S$ be such that $s_1 \leq t_1$ and $s_2 \leq t_2$. Since $S$ is downward directed, there exists $p \in S \subseteq S'$ such that $p \leq s_1$ and $p \leq s_2$. Thus $p \leq t_1$ and $p \leq t_2$, showing that $S'$ is downward directed.

Since for all $t\in S'$ there exists $s \in S$ with $s \leq t$, it follows that $\mathrm{glb} (S) \leq t$. Thus $\mathrm{glb} (S)$ is a lower bound for $S'$. On the other hand, since $S\subseteq S'$, any lower bound for $S'$ must be $\leq \mathrm{glb}(S)$. Hence $\mathrm{glb} (S) = \mathrm{glb}(S')$.

(2) If $S'$ were to have a least element $s$, then $s \in S$, by the definition of $S'$, and  $s = \mathrm{glb}(S') = \mathrm{glb}(S)$ would also be the least element of $S$.

Now, let $s \in S$ be any element. Assuming that $S$ has no least element, there exists $p \in S$ such that $s\not\leq p$. Since $S$ is downward directed, there exists $t \in S$ such that $t \leq s$ and $t \leq p$. Finally, $s\not\leq p$ implies that $t \neq s$, and hence $t<s$.
\end{proof}

The poset $S'$ constructed in the previous lemma happens to be a \emph{filter} in $P$, i.e., a downward directed subset with the property that for all $s \in S'$ and $t \in P$, $s \leq t$ implies that $t \in S'$. In fact, $S'$ is the filter of $P$ \emph{generated} by $S$, i.e., the smallest filter containing $S$.

\begin{lemma} \label{glb-lemma}
Let $\, (P, \leq)$ be a partially ordered set, and let $S \subseteq \AT(P)$ be downward directed with no least element. Then there exists $S'' \subseteq P$ downward directed with no least element, such that $\, \mathrm{glb}(S) = \mathrm{glb}(S'')$ and for all $s'' \in S''$ there exists $s \in S$ satisfying $s \leq_\AT s''$.
\end{lemma}

\begin{proof}
We observe that $\mathrm{glb}(S) \in \AT(P)$, since $(\AT(P), \leq_\AT)$ satisfies GLB, by Theorem~\ref{EPspec} and Proposition~\ref{primeintersect}.

We begin by setting $$S' = S \cup \{t \in \AT(P) \mid \exists s \in S \text{ such that } s \leq_\AT t\}.$$ Then, by Lemma~\ref{largedd}, $S'$ is downward directed, $\mathrm{glb}(S) = \mathrm{glb}(S')$, and $S'$ has no least element. Now let $S'' = S' \cap P$. To see that $S''$ is also downward directed, let $s,t \in S''$. Then there exists $r \in S'$ such that $r \leq_\AT s$ and $r \leq_\AT t$, since $S'$ is downward directed. If $r \in P$, then $r \in S''$, as required. We  therefore assume that $r \in \AT(P) \setminus P$, and hence that $r = x_{[T]}$ for some downward directed subset $T$ of $P$ with no least element. Since $x_{[T]} \leq_\AT s$ and $x_{[T]} \leq_\AT t$, there exist $p,q \in T$ such that $p \leq s$ and $q \leq t$, by Definition~\ref{leqAdef}(3). Since $T$ is downward directed, there exists $u \in T \subseteq P$ such that $u \leq p \leq s$ and $u \leq q \leq t$. By the definition of $S'$, since $r \in S'$, there exists $v \in S$ such that $v \leq_\AT r$. Then $u \in S'$, since $v \leq_\AT r = x_{[T]}\leq_\AT u$, and hence $u \in S''$, showing that $S''$ is downward directed. 

Note that $\mathrm{glb}(S') \leq_\AT \mathrm{glb}(S'')$, since $S'' \subseteq S'$. Also, for all $x_{[T]} \in S' \setminus S''$ we have $\mathrm{glb}(S'') \leq_\AT x_{[T]}$, since $x_{[T]} = \mathrm{glb}(T)$ and $T\subseteq S''$. Hence $\mathrm{glb}(S'') \leq_\AT \mathrm{glb}(S')$, and therefore $\mathrm{glb}(S) = \mathrm{glb}(S') = \mathrm{glb}(S'')$. Also, $S''$ has no least element since if there were such an element $s \in S''$, then we would have $s = \mathrm{glb}(S'') = \mathrm{glb}(S')$, implying that $S'$ has a least element. Finally, for all $s'' \in S''$ there exists $s \in S$ satisfying $s \leq_\AT s''$, since $S''\subseteq S'$.
\end{proof} 

\begin{proposition} \label{R(A(P))equalsP}
Let $P$ be any partially ordered set. Then $\RT(\AT(P)) = P$.
\end{proposition}

\begin{proof}
Let $S \subseteq \AT(P)$ be a downward directed subset with no least element. Then, by Lemma~\ref{glb-lemma}, there is a downward directed subset $S''$ of $P$ with no least element, such that $\mathrm{glb}(S) = \mathrm{glb}(S'')$. Since $S'' \subseteq P$, we have $\mathrm{glb}(S'') = x_{[S'']}$, and hence $\mathrm{glb}(S'') = \mathrm{glb}(S) \notin P$. As $S \subseteq \AT(P)$ was arbitrary, this implies that $P \subseteq \RT(\AT(P))$. But the reverse inclusion $\RT(\AT(P)) \subseteq P$ holds as well, since every element of $\AT(P)\setminus P$ is the greatest lower bound of a downward directed subset with no least element, by Definition~\ref{leqAdef}. Thus $\RT(\AT(P)) = P$.
\end{proof}

\begin{lemma} \label{APcor}
Let $\, (P, \leq)$ be a partially ordered set. Then $\, (\AT(P), \leq_\AT)$ satisfies \emph{GLB}, \emph{DC}, and \emph{DD}. 
\end{lemma}

\begin{proof}
For any ring $R$, $(\Spec(R), \subseteq)$ satisfies GLB, by Proposition~\ref{primeintersect}. Hence, $(\AT(P), \leq_\AT)$ satisfies GLB, by Theorem~\ref{EPspec}.

To show that $(\AT(P), \leq_\AT)$ satisfies DD, let $S \subseteq \AT(P)$ be downward directed. We wish to find a downward directed subset $T$ of $\RT(\AT(P))$ such that $\mathrm{glb}(S) = \mathrm{glb}(T)$. By Lemma~\ref{glb-lemma} and Proposition~\ref{R(A(P))equalsP}, we may assume that $S$ has a least element $s$. If $s \in P$, then we may take $T = \{s\}$, which is in $\RT(\AT(P))$, by Proposition~\ref{R(A(P))equalsP}. Otherwise $s = x_{[T]}$ for some downward directed subset $T$ of $P = \RT(\AT(P))$ with no least element, by Definition~\ref{leqAdef}, and hence $\mathrm{glb}(S) = x_{[T]} = \mathrm{glb}(T)$.

Finally, to show that $(\AT(P), \leq_\AT)$ satisfies DC, let $S \subseteq \AT(P)$ be downward directed, and let $p \in \RT(\AT(P)) = P$ be such that $\mathrm{glb}(S) \leq_\AT p$. We wish to find $s \in S$ such that $s \leq_\AT p$. If $S$ has a least element $s$, then $s = \mathrm{glb}(S)$, from which the desired conclusion follows. Therefore,  assume that $S$ has no least element. Then $\mathrm{glb}(S) \in \AT(P) \setminus P$, since $ \RT(\AT(P)) = P$ (again by Proposition~\ref{R(A(P))equalsP}),  and hence $\mathrm{glb}(S) = x_{[S'']}$ for some downward directed subset $S'' \subseteq P$ with no least element. Moreover, by Lemma~\ref{glb-lemma}, we may assume that for every $s'' \in S''$ there exists $s \in S$ such that $s \leq_\AT s''$. Thus $s'' \leq p$ for some $s'' \in S''$, by Definition~\ref{leqAdef}(3), and hence $s \leq_\AT p$ for some $s \in S$.
\end{proof}

We now present a description of those partially ordered sets which arise as $\AT(Q)$ for some partially ordered set $Q$.  The description will be given from three points of view: in terms of the poset operations $\AT$ and $\RT$,  in terms of germane poset properties, and in terms of prime spectra of Leavitt path algebras.   

\begin{theorem} \label{APprop}
The following are equivalent for any partially ordered set $\, (P, \leq)$.
\begin{enumerate}
\item[$(1)$] $(P, \leq) \cong (\AT(\RT(P)), \leq_\AT)$.
\item[$(2)$] $(P, \leq) \cong (\AT(P'), \leq'_\AT)$ for some partially ordered set $\, (P', \leq')$.
\item[$(3)$] $(P, \leq) \cong (\Spec(L_K(E_{P'})), \subseteq)$ for any field $K$ and some partially ordered set $\, (P', \leq')$.
\item[$(4)$] $(\RT(\AT(P)), \leq) \cong (\AT(\RT(P)), \leq_\AT)$.
\item[$(5)$] $(P, \leq)$ satisfies \emph{GLB}, \emph{DC}, and \emph {DD}.
\end{enumerate}
\end{theorem}

\begin{proof}
(1) implies (2) tautologically, (2) and (3) are equivalent by Theorem~\ref{EPspec}, (1) and (4) are equivalent by Proposition~\ref{R(A(P))equalsP}, and (2) implies (5) by Lemma~\ref{APcor}. So it suffices to assume (5), and show that (1) holds.

We define the map $\psi : P \to \AT(\RT(P))$ by setting
$$\psi (p)=\left\{
\begin{array}{ll}
p & \mbox{if } p \in \RT(P)\\
x_{[S]} & \mbox{if } p =  \mathrm{glb}(S) \text{ for } S \subseteq \RT(P) \text{ downward directed without least element}  
\end{array}\right..$$
To show that $\psi$  is well-defined, note that if $p \in P \setminus \RT(P)$, then $p = \mathrm{glb}(S)$ for some downward directed $S \subseteq P$ with no least element. Since $(P, \leq)$ satisfies DD, $S$ may be chosen such that $S \subseteq \RT(P)$. Also since $(P, \leq)$ satisfies DC, it is easy to see that if $T$ is another downward directed subset of $\RT(P)$ such that $p = \mathrm{glb}(T)$, then $[T] = [S]$. It follows, by Definition~\ref{leqAdef}, that $\psi$ is well-defined and injective. From the definition of $\AT(\RT(P))$ and the fact that $(P, \leq)$ satisfies GLB it is also evident that $\psi$ is surjective. Thus it remains to show that $\psi$ is an order-embedding.  As in the proof of Claim 4 of Theorem \ref{EPspec}, there are four cases to check.  

Let $p,q \in P$. First, if  $p,q \in \RT(P)$, then $p = \psi(p)$ and $q = \psi(q)$, and hence $p \leq q$ if and only if $\psi(p) \leq_\AT \psi(q)$. 
\noindent
Second, if $p \in \RT(P)$ and $q \in P \setminus \RT(P)$, then, as above, $q = \mathrm{glb}(S)$ for some downward directed $S \subseteq \RT(P)$ with no least element. Thus 
\begin{eqnarray*}
p \leq q & \Leftrightarrow & p \leq s \text{ for all } s \in S\\
& \Leftrightarrow & p \leq_\AT x_{[S]}\text{ (by Definition~\ref{leqAdef})}\\
& \Leftrightarrow & \psi(p) \leq_\AT  \psi(q).
\end{eqnarray*}
\noindent
Third, suppose that $p \in P \setminus \RT(P)$ and $q \in \RT(P)$. Again, $p = \mathrm{glb}(S)$ for some downward directed $S \subseteq \RT(P)$ with no least element. Then 
\begin{eqnarray*}
p \leq q & \Leftrightarrow & s \leq q \text{ for some } s \in S  \text{ (since } (P, \leq) \text{ satisfies DC)}\\
& \Leftrightarrow & x_{[S]} \leq_\AT q  \text{ (by Definition~\ref{leqAdef})}\\
& \Leftrightarrow & \psi(p) \leq_\AT \psi(q).
\end{eqnarray*}
\noindent
Fourth and finally, suppose that $p, q \in P \setminus \RT(P)$, and write $p = \mathrm{glb}(S)$,  $q = \mathrm{glb}(T)$ for some downward directed $S, T \subseteq \RT(P)$ with no least element. Then 
\begin{eqnarray*}
p \leq q & \Leftrightarrow & \text{for all } t \in T \text{ there exists } s \in S \text{ such that }
s \leq t \text{ (since } (P, \leq) \text{ satisfies DC)}\\
& \Leftrightarrow & x_{[S]} \leq_\AT x_{[T]}  \text{ (by Definition~\ref{leqAdef})}\\
& \Leftrightarrow & \psi(p) \leq_\AT \psi(q).
\end{eqnarray*}
\noindent
Hence $\psi : P \to \AT(\RT(P))$ is an order-embedding, as desired.
\end{proof}

In particular, we may recast the implication (5) $\Rightarrow$ (3) of  Theorem  \ref{APprop} as a partial answer to the aforementioned Realization Question.   

\begin{corollary}\label{AnstoRealizationQ}
Let $\, (P,\leq)$ be a partially ordered set which satisfies conditions \emph{GLB}, \emph{DC}, and \emph{DD}, and let $K$ be any field.  Then there exists a $K$-algebra $A$ for which $\, (\Spec(A), \subseteq) \cong (P,\leq).$
\end{corollary}

Corollary~\ref{AnstoRealizationQ} is fairly far-reaching, but not complete.  First, we  know by Proposition~\ref{primeintersect} that condition GLB is a necessary condition in any well-posed Realization Question. Second, it is an open question as to whether condition DD is necessary in general. (Further discussion is presented in Section~\ref{closingremarks}.) Third, a reconsideration of  the graph $E$ of Example~\ref{ExampleSpecnotDC} above  yields  a  Leavitt path algebra $L_K(E)$ for which $\Spec(L_K(E))$  does not satisfy property DC.  Specifically,  we see that  the descending chains $S_M = \{\langle E^0 \setminus M_i \rangle \mid i\in \Z^+\}$ and $S_N = \{\langle E^0 \setminus N_i \rangle \mid i\in \Z^+\}$ in $\Spec(L_K(E))$ have $\mathrm{glb}(S_M) = 0 = \mathrm{glb}(S_N)$, and that each element of $S_M \cup S_N$ is in $\RT(\Spec(L_K(E))$, but no element of $S_M$ is comparable to any element of $S_N$.   

\medskip

We conclude this section with a few observations on how to recognize those elements of a partially ordered set $(P,\leq)$ which belong  to $\RT(P)$, since these feature prominently in the results above. Our focus is on the case where $(P,\leq)$ is the prime spectrum of a Leavitt path algebra.

\begin{lemma} \label{ddglb}
The following are equivalent for any partially ordered set $\, (P, \leq)$ and $p \in P$.
\begin{enumerate}
\item[$(1)$] $p \notin \RT(P)$.
\item[$(2)$] There exists a downward directed $S \subseteq P$ such that $p = \mathrm{glb}(S)$ but $p \notin S$.
\end{enumerate}
\end{lemma}

\begin{proof}
(1) $\Rightarrow$ (2) If $p \notin \RT(P)$, then there exists a downward directed $S \subseteq P$ with no least element such that $p = \mathrm{glb}(S)$. Then $p \notin S$, since otherwise $S$ would have a least element.

(2) $\Rightarrow$ (1) Let $S \subseteq P$ be as in (2). If $S$ were to have a least element $p'$, then $p' \leq \mathrm{glb}(S) = p$, and hence $p' = p$. But this would contradict $p \notin S$, and hence $S$ cannot have a least element.
\end{proof}

\begin{lemma} \label{ddglb2}
Let $R$ be a ring and $I \in\Spec(R)\setminus \RT(\Spec(R))$. Then $$I = \bigcap \{J \in \Spec(R) \mid I \subsetneq J\}.$$
\end{lemma}

\begin{proof}
Let $T = \{J \in \Spec(R) \mid I \subsetneq J\}$. By Lemma~\ref{ddglb}, there is a downward directed $S \subseteq \Spec(R)$ such that $I = \mathrm{glb}(S)$ but $I \notin S$. Moreover, $\mathrm{glb}(S) = \bigcap_{J \in S} J$, by Proposition~\ref{primeintersect}. Then $S \subseteq T$, and hence $\bigcap_{J \in T} J \subseteq \bigcap_{J \in S} J = I$. But $I \subseteq \bigcap_{J \in T} J$, and therefore $I = \bigcap_{J \in T} J$, as claimed.
\end{proof}

A prime ideal $I$ that is the intersection of all the prime ideals strictly containing it (as in the previous lemma) is called \emph{not locally closed}. A classification of the locally closed prime ideals in $L_K(E)$, for $E$ finite, is given in~\cite{ABR}.
  
We note that the converse of Lemma \ref{ddglb2} is not true.  For instance, if $R = \Z$, then the prime ideal $\{0\}$ is the intersection of all the nonzero prime ideals of $\Z$, but $\{0\} \in \RT(\Spec(\Z))$, since there are clearly no downward directed sets in $\Spec(\Z)$ which lack a least element.   

\begin{proposition} \label{maxidmpt}
Let $R$ be a ring, $I \subseteq R$ an ideal, and $e \in R\setminus I$ an idempotent. Then there exists an ideal $M_e$ of $R$ maximal with respect to not containing $e$, such that $I \subseteq M_e$. Moreover, $M_e \in \RT(\Spec(R))$.
\end{proposition}

\begin{proof}
Let $S$ be the set of all the ideals of $R$ that contain $I$ but not $e$. Then it is easy to see that the union of any chain of ideals from $S$ is an ideal that contains $I$ but not $e$. Hence, by Zorn's Lemma, there exists an ideal $M_e$ of $R$ maximal with respect to not containing $e$, such that $I \subseteq M_e$.

To show that $M_e$ is prime, suppose that $J_1$ and $J_2$ are ideals of $R$ such that $J_1, J_2 \not\subseteq M_e$. Then $M_e \subsetneq M_e + J_1$ and $M_e \subsetneq M_e + J_2$, which implies that $e \in M_e + J_1$ and $e \in M_e + J_2$, by the maximality of $M_e$. Hence $$e = e^2 \in (M_e + J_1)(M_e + J_2) \subseteq M_e + J_1J_2,$$ and therefore $M_e \subsetneq M_e + J_1J_2$. It follows that $J_1J_2 \not\subseteq M_e$, and thus $M_e$ is prime. 

Finally, let $T = \{J \in \Spec(R) \mid M_e \subsetneq J\}$. Then $e \in J$ for all $J \in T$, by the maximality of $M_e$. Hence $e \in \bigcap_{J \in T} J$, and therefore $M_e \subsetneq \bigcap_{J \in T} J$. Thus, by Lemma~\ref{ddglb2}, we conclude that $M_e \in \RT(\Spec(R))$.
\end{proof}

\begin{proposition}
Let $K$ be a field, $E$ a graph, $I \in \Spec(L_K(E))$, and $H = I\cap E^0$. If $E^0 \setminus H = \MT(u)$ for some $u \in E^0$, then $I \in \RT(\Spec(L_K(E)))$. In this situation either $u$ is not regular or it is the source of a cycle.
\end{proposition}

\begin{proof}
Seeking a contradiction, suppose that $E^0 \setminus H = \MT(u)$ but $I \notin \RT(\Spec(L_K(E)))$. Then there exists a downward directed $S \subseteq \Spec(R)$ such that $I = \bigcap_{J \in S} J$ but $I \notin S$, by Lemma~\ref{ddglb} and Proposition~\ref{primeintersect}. For each $J \in S$ let $H_J = J \cap E^0$. Then $H =\bigcap_{J \in S} H_J$, and hence $\MT(u) = \bigcup_{J \in S} (E^0 \setminus H_J)$. Thus, $u \in E^0 \setminus H_J$ for some $J \in S$. Since $E^0 \setminus H_J$ satisfies MT1 (as noted in Section~\ref{IntroSection}), we have $\MT(u) \subseteq E^0 \setminus H_J$. Thus $\MT(u) = E^0 \setminus H_J$, and therefore $H=H_J$.

Since $I \subsetneq J$, by Theorem~\ref{primeclass}, $H=H_J$ is possible only in the following three situations.
\begin{enumerate}
\item[$(1)$] $I = \langle H \cup \{v^H \mid v \in B_H\}\rangle$ and $J =\langle H \cup \{v^H \mid v \in B_H\} \cup \{f(c)\}\rangle$, where $c \in \pth(E) \setminus E^0$ is a WK cycle  having source $u$, and $f(x)$ is an irreducible polynomial in $K[x, x^{-1}]$.
\item[$(2)$] $I = \langle H \cup \{v^H \mid v \in B_H \setminus \{u\}\}\rangle$ and $J =\langle H \cup \{v^H \mid v \in B_H\}\rangle$.
\item[$(3)$] $I = \langle H \cup \{v^H \mid v \in B_H \setminus \{u\}\}\rangle$ and $J =\langle H \cup \{v^H \mid v \in B_H\} \cup \{f(c)\}\rangle$, where $c \in \pth(E) \setminus E^0$ is a WK cycle  having source $u$, and $f(x)$ is an irreducible polynomial in $K[x, x^{-1}]$.
\end{enumerate}
By~\cite[Section 5]{Ranga}, which describes the prime ideals of $L_K(E)$ having the same intersection with $E^0$, in cases (1) and (2) there are no prime ideals strictly between $I$ and $J$, while in case (3) the only prime ideal strictly between $I$ and $J$ is $\langle H \cup \{v^H \mid v \in B_H\}\rangle$. But, by Lemma~\ref{largedd}(2), for every $J \in S$ there must exist $J' \in S$ such that $J' \subsetneq J$ (where necessarily $I \subsetneq J'$), giving the desired contradiction.

As observed in Remark~\ref{HcupstuffcapE0equalsH}, since $I$ is an ideal, $H$ must be saturated. Hence $\MT(u)$ satisfies MT2, as noted in Section~\ref{IntroSection}. It is easy to see that this can happen only if either $u$ is not regular or it is the source of a cycle.
\end{proof}

\section{Condition KAP}\label{closingremarks}

We start this section by revisiting a question mentioned in Section \ref{primessection} (following Theorem~\ref{Kapthrm}), one of the central currently-unresolved questions regarding the prime spectra of arbitrary rings.  

\begin{notation} \label{KAPdef}
We assign the indicated name in case the poset $\, (P, \leq)$ satisfies the following property.  

\medskip

(KAP) \ \  For all $p,q\in P$ such that $p < q$, there exist $p',q' \in P$ such that $p \leq p' < q' \leq q$, 

 \qquad \qquad and no $t\in P$ satisfies $p' <t <q'$.
\hfill $\Box$ 
\end{notation}

\begin{question}\label{KAPquestion}
Is there a ring $R$ such that $\, (\Spec(R), \subseteq)$ does not satisfy \emph{KAP}?  
\end{question}

With the comments made subsequent to Corollary~\ref{AnstoRealizationQ} as context, we ask the following.  

\begin{question}\label{DDquestion}
Is there a ring $R$ such that $\, (\Spec(R), \subseteq)$ does not satisfy \emph{DD}? $($Moreover, if  so, can such an $R$ be a Leavitt path algebra?$)$
\end{question}

Although on first glance the two conditions KAP and DD seem quite different, it turns out that there is a connection between them, one which relates Questions~\ref{KAPquestion} and~\ref{DDquestion} to each other.   We establish this connection in Proposition~\ref{KAPiffDD}.  

In general, conditions KAP and DD are independent of one another, as the next two examples show.

\begin{example}
For each $i \in \Z^-$, let $Z_i$ be a copy of $\Z^-$, let $Z_0$ be a copy of $\Z^-$ with a least element $-\infty$ adjoined, and let $P = {\bigsqcup}_{i \in \Z^-\cup \{0\}} Z_i$.  Define a relation $\leq_P$ on all elements $p,q\in P$ by letting $p\leq_P q$ if and only if one of the following holds:
\begin{enumerate}
\item[$(1)$] $p,q \in Z_i$ for some $i \in \Z^-\cup \{0\}$ and $p \leq q$, 
\item[$(2)$] $p \in Z_0 \setminus \{-\infty\}$ and $q \in Z_r$ for some $p \leq r \leq -1$. 
\end{enumerate}
Pictorially, $P$ can be represented as follows.
$$
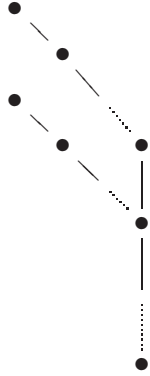
\xymatrix@=.5pc{ 
{\bullet} \ar@{-}[dr] & & & &\\
& {\bullet} \ar@{-}[dr] & & &\\
{\bullet} \ar@{-}[dr] & & \ar@{.}[dr] & & & &\\
& {\bullet} \ar@{-}[dr] & & {\bullet} \ar@{-}[dd]\\
& & \ar@{.}[dr] & \\
& & & {\bullet} \ar@{-}[dd] \\
& & & \\
& & & \ar@{.}[dd] \\
& & & & \\
& & & {\bullet} \\
}$$
Then it is easy to see that $(P,\leq_P)$ is a partially ordered set. Moreover, since every element of $P \setminus \{-\infty\}$ has an immediate successor, $P$ satisfies KAP. Also, $\RT(P) = P \setminus Z_0$, from which it follows that every downward directed subset of $\RT(P)$ is contained in some $Z_i$ ($i<0$). Thus no downward directed subset of $\RT(P)$ has $-\infty$ as the greatest lower bound, and hence $P$ does not satisfy DD.  \hfill $\Box$ 
\end{example}

\begin{example}
Let $[0,1] \subseteq \R$ denote the closed interval between $0$ and $1$ on the real number line. For each $i \in [0,1]$, let $Z_i$ be a copy of $\Z^-$, and let $P = [0,1] \ {\sqcup} \  {\bigsqcup}_{i \in [0,1]} Z_i$. Define a relation $\leq_P$ on all elements $p,q\in P$ by letting $p\leq_P q$ if and only if one of the following holds:
\begin{enumerate}
\item[$(1)$] $p,q \in Z_i$ for some $i \in [0,1]$ and $p \leq q$, 
\item[$(2)$] $p,q \in [0,1]$ for some and $p \leq q$, 
\item[$(3)$] $p \in [0,1]$ and $q \in Z_r$ for some $p \leq r \leq 1$.
\end{enumerate}
Then, as in the previous example, $(P,\leq_P)$ is a partially ordered set. Also, $P$ satisfies DD since every element of $P$ is the greatest lower bound of some $Z_i$ or subset thereof, and $\RT(P) = P \setminus [0,1] = {\bigsqcup}_{i \in [0,1]} Z_i$. However, $P$ clearly does not satisfy KAP, since it contains $[0,1]$.  \hfill $\Box$
\end{example}

The two previous examples notwithstanding, there is  however a connection between the conditions KAP and DD for partially ordered sets satisfying GLB and a strengthened version of DC. (Note that the partially ordered sets in both of those examples satisfy GLB.)

\begin{proposition}\label{KAPiffDD}
Let $\,(P,\leq)$ be a partially ordered set satisfying \emph{GLB}, and suppose that for every downward directed $S \subseteq P$ and every $p \in P$ satisfying $p > \mathrm{glb}(S)$, we have $p \geq s$ for some $s \in S$. Then $P$ satisfies \emph{KAP} if and only if $P$ satisfies \emph{DD}.
\end{proposition}

\begin{proof}
Suppose that $P$ satisfies KAP. Let $S \subseteq P$ be downward directed without least element, and set $\mathrm{glb}(S) = p$. We wish to find a downward directed subset of $\RT(P)$ whose greatest lower bound is $p$. Upon replacing $S$ with a larger set, by Lemma~\ref{largedd} we may assume that if $t \in P$ and $s \in S$ are such that $s \leq t$, then $t \in S$. We shall show that $S \cap \RT(P)$ is downward directed and has $p$ as the greatest lower bound.

Let $q \in S$ be any element. Then, by Lemma~\ref{largedd}(2), there exists $q' \in S$ such that $q' < q$. Since $P$ satisfies KAP, there exist $q_1,q_2 \in P$ such that $q' \leq q_1 < q_2 \leq q$ and there are no elements of $P$ strictly between $q_1$ and $q_2$. We claim that $q_1 \in \RT(P)$. If not, then there exists $T \subseteq P$ downward directed without least element, such that $\mathrm{glb}(T)  = q_1$. Again applying Lemma~\ref{largedd} we may assume that if $t \in P$ and $s \in T$ are such that $s \leq t$, then $t \in T$. By our hypothesis on $P$, since $q_1 < q_2$, there exists $t \in T$ such that $t \leq q_2$, and hence $q_2 \in T$. But then, by Lemma~\ref{largedd}(2), there exists $s \in T$ such that $q_1 < s < q_2$, contradicting our choice of $q_1$ and $q_2$. Thus $q_1 \in \RT(P)$, and since $q' \leq q_1$, it follows that $q_1 \in S$. We have therefore shown that for every $q \in S$ there exists $q_1 \in S \cap \RT(P)$ such that $q_1 < q$. 

Next, we show that $S \cap \RT(P)$ is downward directed. Given any $s,t \in S \cap \RT(P)$ there must be some $r \in S$ such that $r \leq s$ and $r \leq t$, and hence there must be some $r_1 \in S \cap \RT(P)$ such that $r_1 < r$, and therefore also $r_1 \leq s$ and $r_1 \leq t$. Finally, we note that $\mathrm{glb}(S \cap \RT(P)) = p$. For $S \cap \RT(P) \subseteq S$ implies that $p \leq \mathrm{glb}(S \cap \RT(P))$. On the other hand if $p \neq \mathrm{glb}(S \cap \RT(P))$, then our assumption on $P$ implies that $q \leq \mathrm{glb}(S \cap \RT(P))$ for some $q \in S$. But then $q_1 < q$ for some $q_1 \in S \cap \RT(P)$, and hence $q_1 < \mathrm{glb}(S \cap \RT(P))$, which is absurd. Therefore $\mathrm{glb}(S \cap \RT(P)) = p$.

\smallskip

Conversely, suppose that $P$ satisfies DD. Let $p,q \in P$ be such that $p< q$. We wish to find $p', q' \in P$ such that $p \leq p' < q' \leq q$, and there are no elements of $P$ strictly between $p'$ and $q'$. First, suppose that $p \in \RT(P)$. Let $S$ be the set of all chains of elements $s$ of $P$ such that $p < s \leq q$. The union of any chain of chains in $S$ is again a chain in $S$, and hence, by Zorn's lemma there is a maximal chain $C \in S$. Since $C \subseteq P$ is downward directed, $p \notin C$, and $p \in \RT(P)$, we have $p \neq \mathrm{glb}(C)$, by Lemma~\ref{ddglb}. Also, by the maximality of $C$, there cannot be an element of $P$ strictly between $p$ and $\mathrm{glb}(C)$. Hence, letting $p=p'$ and $\mathrm{glb}(C) = q'$, we have the desired conclusion.

We may therefore assume that $p \notin \RT(P)$. Since $P$ satisfies DD, there exists a downward directed $S \subseteq \RT(P)$ without least element, such that $\mathrm{glb}(S) = p$. By our hypothesis on $P$, since $p < q$, there exists $p' \in S$ such that $p' \leq q$. Moreover, by Lemma~\ref{largedd}(2), we may assume that $p' < q$. Since $p' \in \RT(P)$, by the argument in the previous paragraph, there exists $q' \in P$ such that $p' < q' \leq q$ and there are no elements of $P$ strictly between $p'$ and $q'$, again giving the desired conclusion.
\end{proof}

\begin{remark}\label{DCvsDCplus}
  A moment's reflection yields that DC is equivalent to a modified version of DC, in which the requirement that $p\geq \mathrm{glb}(S)$ is replaced by $p > \mathrm{glb}(S)$. Thus the second hypothesis on $P$ in Proposition~\ref{KAPiffDD} is almost identical to condition DC; the only difference is that $p\in \RT(P)$ in DC, while $p \in P$ in the Proposition.     So the condition of the Proposition is stronger than DC.   Indeed, consider the poset $P$ having the following Hasse diagram. 
  $$
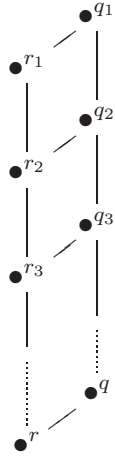
\xymatrix@=.5pc{ 
  & \bullet^{q_1} \\
  \bullet^{r_1} \ar@{-}[ru] \ar@{-}[dd] & \\
  & \bullet^{q_2} \ar@{-}[uu] \\
  \bullet^{r_2} \ar@{-}[ur] & \\
  & \bullet^{q_3} \ar@{-}[uu]  \ar@{-}[dd]   \\
  \bullet^{r_3} \ar@{-}[uu] \ar@{-}[ur]  {}  
  \ar@{-}[dd]   & \\
 &  \\
  &  \\
  &  \bullet^q \ar@{.}[uu]  \\
\bullet^r \ar@{.}[uu] \ar@{-}[ur] & \\
 & \\  }$$
If $S_r = \{r_i \ | \ i\in \Z^+\}$ and $S_q = \{q_i \ | \ i\in \Z^+\}$, then $r = {\rm glb}(S_r)$ and $q = {\rm glb}(S_q)$.   We see that $\RT(P) = S_r \cup S_q$, so that every $p \in \RT(P)$ indeed has the property specified in condition DC.   However, if we consider $q \in P$, then $q > {\rm glb}(S_r)$, but there is no $r_i \in S_r$ for which $q \geq r_i$.   So $P$ does not satisfy the second hypothesis of Proposition~\ref{KAPiffDD}.    \hfill $\Box$
\end{remark}

We have, by Proposition~\ref{primeintersect}, that $\Spec(R)$ satisfies GLB for any ring $R$.  Furthermore, it is clear that  any totally ordered set $P$ satisfies  the second hypothesis of Proposition~\ref{KAPiffDD}. So we get the following. 

\begin{corollary}\label{KAPiffDDintotalorder}  Let $R$ be a ring for which $\,\Spec(R)$ is totally ordered.   Then $\, \Spec(R)$  satisfies condition \emph{KAP} if and only if $\, \Spec(R)$ satisfies condition \emph{DD}.     
\end{corollary} 

The  difficulty in determining an answer to Question~\ref{KAPquestion}  persists even when a key additional assumption is made on the underlying spectra.

\begin{question}\label{KAPfortotalorder}
Is there a ring $R$ such that $\, (\Spec(R), \subseteq)$ is totally ordered, and for which $\, (\Spec(R), \subseteq)$ does not satisfy \emph{KAP}?   

Equivalently, by Corollary~\ref{KAPiffDDintotalorder}:  Is there a  ring $R$ such that $\, (\Spec(R), \subseteq)$ is totally ordered, and for which $\, (\Spec(R), \subseteq)$ does not satisfy \emph{DD}?    
\end{question}

We ask one final  question regarding the prime spectrum of a Leavitt path algebra. It is motivated by the fact that in our main results the graph $E_P$ we associated to an arbitrary poset $P$ is always acyclic and contains no breaking vertices (see Lemma~\ref{chrislemma} and its proof).

\begin{question}
Let $\, (P, \leq)$ be a partially ordered set such that $\, (P, \leq) \cong (\Spec(L_K(E)), \subseteq)$ for some field $K$ and graph $E$. Does there necessarily exist a graph $F$,  where  $F$ is  acyclic and contains no breaking vertices, for which $(P, \leq) \cong  (\Spec(L_K(F)), \subseteq)$?   
\end{question}

\section*{Appendix: greatest lower bounds for chains}

We conclude the paper by showing that the condition GLB on a poset $P$ is equivalent to the condition that $P$ has a greatest lower bound for every chain. The argument is due to George Bergman.

\begin{lemma} \label{dd-closure-lemma}
Let $\, (P, \leq)$ be a downward directed poset. Then for any subset $S$ of $P$ there exists a downward directed subset $T$ of $P$ such that $S \subseteq T$ and $\, |T| \leq \mathrm{max}(|S|,\aleph_0)$.
\end{lemma}

\begin{proof}
Since $P$ is downward directed, by the axiom of choice, there is a function $f : P \times P \to P$ such that $f(s,t) \leq s$ and $f(s,t) \leq t$ for all $s,t \in P$. 

For every $n\in \N$ define inductively a subset $T_n \subseteq P$ as follows. Let $T_0 = S,$ and assuming that $T_n$ is defined for some $n \in \N$, let $$T_{n+1} = T_n \cup \{f(s,t) \mid s,t \in T_n\}.$$ Note that $|T_{n+1}| \leq |T_n| + |T_n| \cdot |T_n|$, and so $|T_{n+1}| = |T_n|$ if $T_n$ is infinite, and $|T_{n+1}| < \aleph_0$ if $T_n$ is finite, for each $n \in \N$. Hence $|T_n| \leq \mathrm{max}(|S|,\aleph_0)$ for each $n \in \N$.

Now letting $T = \bigcup_{n \in \N} T_n$, we have $S \subseteq T$ and $|T| \leq \mathrm{max}(|S|,\aleph_0)$. Finally, $T$ is downward directed, since for all $s,t \in T$ there is some $n \in \N$ such that $s,t \in T_n$, and hence $f(s,t)\in T_{n+1} \subseteq T$.
\end{proof}

\begin{lemma} \label{dd-union-lemma}
Let $\, (P, \leq)$ be a partially ordered set and $S$ an uncountable downward directed subset of $P$. Then $S$ can be written as the union of a chain of downward directed subsets of $S$, all having cardinality less than $\, |S|$.
\end{lemma}

\begin{proof}
Let $|S| = \kappa$, and write $S = \{s_\alpha \mid \alpha \in \kappa\}$. We will construct for each $\alpha \in \kappa$ a downward directed subset $S_\alpha$ of $S$ satisfying following conditions:
\begin{enumerate}
\item[$(1_\alpha)$] $|S_\alpha| \leq \mathrm{max}(|\alpha|,\aleph_0)$,
\item[$(2_\alpha)$] $s_{\alpha'} \in S_{\alpha}$ for all $\alpha' < \alpha$,
\item[$(3_\alpha)$] $S_{\alpha'}\subseteq S_{\alpha}$ for all $\alpha' < \alpha$.
\end{enumerate}

Set $S_0 = \{s_0\}$, and assume inductively that there exists $\alpha \in \kappa$ such that for all $\beta <\alpha$ there are downward directed subsets $S_\beta$ of $S$ satisfying $(1_\beta), (2_\beta), (3_\beta)$. If $\alpha$ is a limit ordinal, let $S_\alpha = \bigcup_{\beta < \alpha} S_\beta$. Then $|S_\alpha| \leq \mathrm{max}(|\alpha|,\aleph_0)$, since each $S_\beta$ satisfies $(1_\beta)$, and $S_\alpha$ is downward directed, since each $S_\beta$ satisfies $(3_\beta)$ and is downward directed. If $\alpha$ is a successor ordinal, then by Lemma~\ref{dd-closure-lemma}, we can find a downward directed subset $S_\alpha$ of $S$ containing $S_{\alpha-1} \cup \{s_{\alpha-1}\}$ and satisfying $$|S_\alpha| \leq \mathrm{max}(|S_{\alpha-1} \cup \{s_{\alpha-1}\}|,\aleph_0) \leq \mathrm{max}(|\alpha-1|+1,\aleph_0) \leq \mathrm{max}(|\alpha|,\aleph_0).$$ It follows that in either case, $S_\alpha$ is downward directed and satisfies $(1_\alpha), (2_\alpha), (3_\alpha)$, as desired.

Finally, $S = \bigcup_{\alpha \in \kappa} S_\alpha$, since each $S_\alpha$ satisfies $(2_\alpha)$ and is a subset of $S$. Moreover, the $S_\alpha$ form a chain, by $(3_\alpha)$, and each has cardinality less than $\kappa$, by $(1_\alpha)$.
\end{proof}

\begin{notation}
We assign the indicated name in case the poset $\, (P, \leq)$ satisfies the following property.  

\medskip

(GLBC) \ \  Every chain in $P$ has a greatest lower bound in $P$.
\hfill $\Box$ 
\end{notation}

\begin{theorem}[Bergman] \label{bergman}
Let $\, (P, \leq)$ be a partially ordered set. Then $P$ satisfies \emph{GLBC} if and only if it satisfies \emph{GLB}.
\end{theorem}

\begin{proof}
Clearly, GLB implies GLBC for any poset $P$. Thus let us assume that $P$ satisfies GLBC, and prove that it satisfies GLB.

First, let $S$ be a countable downward directed subset of $P$. We wish to show that $S$ has a greatest lower bound in $P$. This is obvious if $S$ is finite, so let us assume that $|S| = \aleph_0$ and write $S = \{s_n \mid n \in \N\}$. Define inductively for each $n \in \N$ an element $t_n \in S$ as follows. Set $t_0 = s_0$, and assuming that $t_n$ is defined for some $n \in \N$, let $t_{n+1} \in S$ be such that $t_{n+1} \leq t_n$ and $t_{n+1} \leq s_n$. Then $T=\{t_n \mid n \in \N\} \preceq S$. Since $P$ satisfies GLBC, the chain $T$ has a greatest lower bound in $P$. Moreover, $\mathrm{glb}(T)$ is a lower bound for $S$ since $T \preceq S$, and hence $\mathrm{glb}(T) = \mathrm{glb}(S)$, since $T \subseteq S$. In particular, we have established the result in the case where $P$ is countable.

Next, assume that $\aleph_0 < |P|$, and let $\kappa \leq |P|$ be an uncountable cardinal. Assume inductively that every downward directed subset of $P$ with cardinality less than $\kappa$ has a greatest lower bound in $P$. Let $S$ be a downward directed subset of $P$ of cardinality $\kappa$. Then, by Lemma~\ref{dd-union-lemma}, we can write $S = \bigcup_{\beta \in \kappa} S_\beta$, where each $S_\beta$ is downward directed and has cardinality less than $\kappa$, and where $S_{\beta'}\subseteq S_{\beta}$ for all $\beta' < \beta$. By the inductive hypothesis, for each $\beta \in \kappa$, the downward directed set $S_\beta$ has a greatest lower bound $s_\beta \in P$. Moreover, since the $S_\beta$ form a chain (under set inclusion), the elements $s_\beta$ also form a chain (under $\leq$), and $T=\{s_\beta \mid \beta \in \kappa\} \preceq S$. Since $P$ satisfies GLBC, the chain $T$ has a greatest lower bound in $P$. Moreover, $\mathrm{glb}(T)$ is a lower bound for $S$ (since $T \preceq S$), and hence $\mathrm{glb}(T) = \mathrm{glb}(S)$, since $S = \bigcup_{\beta \in \kappa} S_\beta$.

Thus, by induction, every downward directed subset of $P$ has a greatest lower bound. That is, $P$ satisfies GLB.
\end{proof}

In view of this fact, it is natural to ask whether our main results (Theorems~\ref{EPspec} and~\ref{APprop}) could be recast entirely in terms of chains, without any reference to arbitrary downward directed sets. To explore this question briefly, we define a ``chain analogue" of $\AT$.

\begin{definition}
Given a partially ordered set $\, (P, \leq)$, let 
$$\ATC(P) = P \cup \{x_{[S]} \mid S \subseteq P \text{ is a chain without least element}\}.$$ Further, extend $\, \leq$ to a binary relation $\, \leq_{\ATC}$ on $\ATC(P)$ by letting $p \leq_{\ATC} q$ $(p,q \in \ATC(P))$ if any of the conditions $\, (1)$--$(4)$ in Definition~\ref{leqAdef} holds $($with $\ATC(P)$ in place of $\AT(P)$$)$. \hfill $\Box$
\end{definition}

\begin{lemma}
$(\ATC(P), \leq_{\ATC})$ is a poset for any poset $\, (P, \leq)$.
\end{lemma}

\begin{proof}
Let $\phi : \ATC(P) \to \AT(P)$ be the natural embedding; that is, the map that sends the copy of $P$ in $\ATC(P)$ to that in $\AT(P)$, and sends each $x_{[S]}$ to the corresponding element in $\AT(P)$. Then clearly, $p \leq_{\ATC} q$ if and only if $\phi(p) \leq_{\AT} \phi(q)$ for all $p, q \in \ATC(P)$. Since, by Theorem~\ref{EPspec}, $(\AT(P), \leq_{\AT})$ is a partially ordered set, $(\phi(\ATC(P)), \leq_{\AT})$ is one as well. Thus, the same holds for $(\ATC(P), \leq_{\ATC})$.
\end{proof}

On the basis of Theorem~\ref{bergman} and the previous lemma, one might guess that $\AT(P)$ is order-isomorphic to $\ATC(P)$. However, that is not true in general, as shown in the next example, which is based on one from~\cite{Diestel}. 

\begin{example}
Let $P_0$ be an uncountable antichain (i.e., a poset where no pair of elements is comparable). Assuming that $P_n$ is defined for $n \geq 0$, let $P_{n+1}$ be obtained by adding for every pair of points $p,q \in P_n$ a new point less than both (and hence also less than every point greater than either $p$ or $q$). Then $P = \bigcup_{n=1}^\infty P_n$ is partially ordered, by the ordering $\leq$ inherited from the $P_n$. It is also immediate that $P$ is downward directed. 

Note that each element of $P$ is less than only finitely many others, and that every totally ordered subset of $P$ is countable. From this it follows that $P$ does not have a greatest lower bound in $\ATC(P)$. For, any such greatest lower bound would be of the form $x_{[S]}$ for some chain $S \subseteq P$ without least element. But, since $S$ is countable, as noted above, there are only countably many elements $p \in P$ such that $x_{[S]} \leq_{\ATC} p$, by the definition of $\leq_{\ATC}$. Hence $x_{[S]}$ cannot be the greatest lower bound for $P$, since $P$ is uncountable.

On the other hand, by Theorem~\ref{EPspec}, $(\AT(P), \leq_\AT)$ satisfies GLB, and therefore cannot be order-isomorphic to $(\ATC(P), \leq_\ATC)$.  \hfill $\Box$
\end{example}

This example shows that it is not possible to replace $\AT$ with $\ATC$ in Theorems~\ref{EPspec} and~\ref{APprop}, and therefore one cannot remove all references to downward directed sets (replacing them with chains) in our main results. Still, similarly to how we modified GLB to define GLBC, one could define ``chain analogues" of DC and DD, and then ask whether (or under what additional hypotheses) those are equivalent to DC and DD. Considering such questions at length, however, is beyond the scope of this paper.

\vspace{.1in}

\noindent
Department of Mathematics, University of Colorado, Colorado Springs, CO, 80918, USA \newline
\noindent {\href{mailto:abrams@math.uccs.edu}{abrams@math.uccs.edu}}

\vspace{.1in}

\noindent
Department of Algebra, Geometry and Topology, University of M\'{a}laga, M\'{a}laga, 29071, Spain \newline
\noindent {\href{mailto:g.aranda@uma.es}{g.aranda@uma.es}}

\vspace{.1in}

\noindent
Department of Mathematics, University of Colorado, Colorado Springs, CO, 80918, USA \newline
\noindent {\href{mailto:zmesyan@uccs.edu}{zmesyan@uccs.edu}}

\vspace{.1in}

\noindent
Department of Mathematics, University of Colorado, Colorado Springs, CO, 80918, USA \newline
\noindent {\href{mailto:cdsmith@gmail.com}{cdsmith@gmail.com}}


\begin{thebibliography}{00}

\bibitem{ASurvey}  G.\ Abrams, \emph{Leavitt path algebras: the first decade},  Bull. \ Math.\ Sci.\ \textbf{5}(1) (2015) 59--120. 

\bibitem{ABR} G.\ Abrams, J.\ P.\ Bell, and K.\ M.\ Rangaswamy, \emph{The Dixmier-Moeglin equivalence for Leavitt path algebras,} Algebr.\ Represent.\ Theor.\ \textbf{15}(3) (2012) 407--425.

\bibitem{APPSM}  G.\ Aranda Pino, E.\ Pardo, and M.\ Siles Molina, \emph{Prime spectrum and primitive Leavitt path algebras,}  Indiana\ Univ.\ Math. J.\ \textbf{58}(2) (2009) 869--890.

\bibitem{Diestel} R.\ Diestel, \emph{Relating subsets of a poset, and a partition theorem for WQOs,} Order \textbf{18}(3) (2001) 275--279.

\bibitem{Facchini} A.\ Facchini, \emph{Generalized Dedekind domains and their injective modules,} J.\ Pure Appl.\ Algebra \textbf{94}(2) (1994) 159--173.

\bibitem{GRV} B.\ Greenfeld, L.\ H.\ Rowen, and U.\ Vishne, \emph{Unions of chains of primes,}  J.\ Pure Appl.\ Algebra \textbf{220}(4) (2016) 1451--1461.

\bibitem{HW} W.\ Heinzer and S.\ Wiegand, \emph{Prime ideals in two-dimensional polynomial rings,} Proc.\ Amer.\ Math.\ Soc.\ \textbf{107}(3) (1989) 577--586.

\bibitem{Hochster} M.\ Hochster, \emph{Prime ideal structure in commutative rings,} Trans.\ Amer.\ Math.\ Soc.\ \textbf{142} (1969) 43--60. 

\bibitem{Kap} I.\ Kaplansky, \emph{Commutative Rings, Revised Edition,} University of Chicago Press, Chicago and London, 1974.

\bibitem{Lewis} W.\ J.\ Lewis, \emph{The spectrum of a ring as a partially
ordered set,} J.\ Algebra \textbf{25} (1973) 419--434.

\bibitem{LO} W.\ J.\ Lewis and J.\ Ohm, \emph{The ordering of Spec \rm{R},} Can.\ J.\ Math.\ \textbf{28}(3) (1976) 820--835.

\bibitem{Passman} D.\ Passman, \emph{Prime ideals in normalizing extensions,} J.\ Algebra \textbf{73}(2) (1981) 556--572.

\bibitem{Ranga} K.\ M.\ Rangaswamy, \emph{The theory of prime ideals of Leavitt path algebras over arbitrary graphs,} J.\ Algebra \textbf{375} (2013) 73--96.

\bibitem{Sarussi} S.\ Sarussi, \emph{Totally ordered sets and the prime spectra of rings,} Comm.\ Alg.\ \textbf{45}(1) (2017) 411--419.

\bibitem{Speed} T.\ P.\ Speed, \emph{On the order of prime ideals,} Algebra Universalis \textbf{2} (1972) 85--87.

\bibitem{Wiegand} R.\ Wiegand, \emph{The prime spectrum of a two-dimensional affine domain,} J.\ Pure Appl.\ Algebra \textbf{40}(2) (1986) 209--214.

\end{thebibliography}
\end{document}